\documentclass[final,leqno]{siamltex}
\usepackage{xcolor}
\usepackage[normalem]{ulem}

\usepackage{amssymb}
\usepackage{amsmath}
\usepackage{graphicx}
\usepackage{url}
\usepackage{pdflscape}
\usepackage{comment}
\usepackage{wrapfig}
\usepackage{verbatim}
\usepackage{caption}
\usepackage{subcaption}
\usepackage{booktabs} 
\usepackage{float} 
\usepackage{multicol} 
\usepackage{colortbl}
\usepackage{enumitem}
\usepackage{longtable}

\usepackage{algorithm}
\usepackage{algorithmic}
\usepackage[left=3.25cm,top=3.1cm,right=3.0cm,bottom=3.25cm,bindingoffset=0.5cm]{geometry}
\usepackage[colorinlistoftodos,prependcaption,textsize=tiny]{todonotes}
\RequirePackage{fix-cm}

\newtheorem{assumption}[theorem]{Assumption}

\newtheorem{remark}[theorem]{Remark}
\newtheorem{example}[theorem]{Example}

\usecounter{algorithmenumi}
\providecommand{\norm}[1]{\lVert#1\rVert}

\DeclareMathOperator{\range}{range}     
\DeclareMathOperator{\E}{\mathbf{E}}
\DeclareMathOperator{\Exp}{\mathbf{E}}
\DeclareMathOperator{\Prob}{\mathbf{P}}

\newcommand\tagthis{\addtocounter{equation}{1}\tag{\theequation}}
\newcommand{\ve}[2]{\langle #1 ,  #2 \rangle}

\newcommand{\cR}{\mathcal{R}}
\newcommand{\R}{\mathbb{R}}

\newcommand{\RSD}{RSD }

\newcommand{\sS}{\mathcal{S}}

\graphicspath{{./}{./exp/}}

\title{Randomized sketch  descent methods for non-separable linearly constrained optimization}

\author{Ion Necoara,  Martin Tak\'a\v{c}\thanks{I. Necoara is with  Automatic Control and Systems Engineering Department, University  Politehnica Bucharest, 060042 Bucharest, Romania, {\tt\small ion.necoara@acse.pub.ro}. M. Tak\'a\v{c}  is with Industrial and Systems Engineering Department,  Lehigh University,  Bethlehem, PA 18015, USA, {\tt\small Takac.MT@gmail.com}.}}

\begin{document}
\maketitle

\begin{abstract}
In this paper we consider large-scale smooth optimization problems with multiple linear coupled constraints. Due to the non-separability of the  constraints, arbitrary random sketching would not be guaranteed to work. Thus,  we first investigate necessary and sufficient conditions for the sketch  sampling to have well-defined algorithms. Based on these sampling conditions we developed new sketch descent methods for solving general smooth linearly constrained problems, in particular, random sketch descent and accelerated random sketch descent methods. From our knowledge, this is the first convergence analysis of random sketch descent algorithms for optimization problems with multiple  non-separable linear  constraints. For the general case, when the objective function is smooth and non-convex, we prove for the non-accelerated variant sublinear rate in expectation for an appropriate optimality measure. In the smooth convex case, we derive for both algorithms, non-accelerated and accelerated random sketch descent, sublinear convergence rates in the expected values of the objective function. Additionally, if the objective function satisfies a strong convexity type condition, both algorithms converge linearly in expectation. In special cases, where complexity bounds are known for some particular sketching algorithms, such as coordinate descent methods for optimization problems with a single linear coupled constraint, our theory recovers the best-known bounds.  We also show that when  random sketch is sketching the coordinate directions randomly produces better results than the fixed selection rule.  Finally, we present some numerical examples to illustrate the performances of our new algorithms.
\end{abstract}


\section{Introduction}
During the last decade first order methods, that eventually utilize also some curvature information, have become the methods of choice for solving optimization problems of large sizes arising in all areas of human endeavor where data is available, including   machine learning \cite{NecPat:14,RicTac:11,ShaZha:13}, portfolio optimization
\cite{Mar:52,FroRei:15}, internet and multi-agent systems \cite{IshTem:12}, resource allocation \cite{Nec:13,XiaBoy:06}  and image processing \cite{Wri:10}. These large-scale problems are often highly structured (e.g., sparsity in data,  separability in objective function, convexity) and it is important for any optimization method to take advantage of the underlying structure. It turns out that gradient-based algorithms can really benefit from the structure of the optimization models arising in these recent applications  \cite{FerRic:15,Nes:10}.

\textbf{Why random sketch descent methods?}
The optimization problem  we consider in this paper has the following features: {\it the size of data  is very large} so that usual methods based on whole gradient/Hessian computations are
prohibitive; moreover  {\it the constraints are coupled}.  In this case, an appropriate way to approach these problems is through sketch descent methods due to their low memory requirements and low per-iteration computational cost. Sketching is a very general framework that covers as a particular case the (block) coordinate descent methods \cite{LuoTse:93} when  the sketch matrix is given by sampling columns of the identity matrix. Sketching was used, with a big success, to either decrease the computation burden when evaluating the gradient in first order methods \cite{Nes:10} or to avoid solving the full Newton direction in second order methods \cite{pilanci2017newton}.  Another crucial advantage of sketching is that for structured problems  it keeps the computation cost low, while preserving the amount of data brought from RAM to CPU as for full gradient or Newton methods, and consequently allows for better CPUs utilization on modern multi-core machines. Moreover, in many situations general sketching  keeps the per-iteration running-time almost unchanged when compared to the particular sketching of the identity  matrix (i.e. comparable to coordinate descent settings). This, however, leads to a smaller number of iterations needed to achieve the desired quality of the solution as observed e.g. in  \cite{qu2016sdna}.

In second order methods sketching was used to either decrease the computation cost when evaluating the  full Hessian or to avoid solving the full Newton direction. In \cite{pilanci2017newton,berahas2017investigation} a Newton sketch algorithm was proposed for unconstrained  self-concordant minimization, which performs an approximate Newton step, wherein each iteration only a sub-sampled Hessian is used. This procedure significantly reduces the computation cost, and still guarantees superlinear convergence for self-concordant objective functions.  In \cite{qu2016sdna}, a random  sketch method was used to minimize a smooth function which admits a non-separable quadratic upper-bound. In each iteration a block of coordinates was chosen and  a subproblem  involving  a  random  principal
submatrix of the Hessian of the quadratic approximation was solved to obtain an improving direction.

In first order methods particular sketching  was used, by choosing as sketch matrix (block) columns of the identity matrix, in order to avoid computation of the full gradient, leading to coordinate descent framework. The main differences in all variants of coordinate descent methods consist in the criterion of choosing at each iteration the coordinate over which we minimize the  objective function and the complexity of this choice. Two classical criteria used often in these algorithms are the cyclic  and the greedy  coordinate search,
which significantly differs by the amount of computations required
to choose the appropriate index. For cyclic coordinate search
estimates on the rate of convergence were given recently in
\cite{BecTet:12,GurOzd:17,SunYe:16}, while for the greedy coordinate search (e.g. Gauss-Southwell rule) the convergence rates were given  in
\cite{Tse09,LuoTse:93}. Another approach is based on random choice
rule, where the coordinate search is random. Complexity
results on random coordinate descent methods  for smooth convex
objective functions were obtained in \cite{Nes:10,Nec:13}. The
extension to composite convex objective functions was given e.g. in
\cite{RicTac:11,NecCli:14,LuXia:15,NecPat:14,richtarik2016optimal}. These methods are inherently serial. Recently, accelerated \cite{FerRic:15,fercoq2014fast}, parallel \cite{NecCli:14,RicTac:13,tappenden2018complexity}, asynchronous \cite{LiuWri:14} and distributed implementations \cite{takavc2015distributed, marevcek2015distributed} of coordinate descent methods were also analyzed. Let us note that the idea of sketching or sub-sampling was also successfully applied in various other settings, including \cite{wang2017sketching,richtarik2017stochastic,gower2018stochastic}.

\vspace{0.1cm}

\textbf{Related work}. However, most of the aforementioned sketch
descent methods assume essentially unconstrained problems, which at best
allow separable constraints. In contrast, in this paper  we consider sketch
descent methods for general smooth problems with linear coupled constraints.
Particular sketching-based algorithms, such as  greedy coordinate descent schemes,  for solving   linearly constrained optimization problems were investigated in \cite{Tse09,LuoTse:93}, while more recently in \cite{Bec:12} a greedy coordinate descent method is
developed for minimizing a smooth function subject to a single
linear equality constraint and additional bound constraints on the
decision variables. Random coordinate descent methods that choose at least 2 coordinates at each iteration  have been also proposed recently for solving convex problems with a single linear coupled  constraint  in  \cite{Nec:13,NecPat:14,NecNes:11}. In all these papers, detailed convergence analysis is provided for both, convex and non-convex settings. Motivated by the work in \cite{NecNes:11} several recent papers have tried to extended the random coordinate descent settings to multiple linear coupled constraints \cite{FroRei:15,NecPat:14,RedHef:14}. In particular, in \cite{RedHef:14} an extension of the 2-random coordinate descent method from  \cite{NecNes:11} has been analyzed, however under very conservative assumptions, such as full rank condition on each block of the matrix describing the linear constraints.  In \cite{FroRei:15} a particular sketch  descent method is proposed, where the sketch matrices  specify arbitrary subspaces that need to generate the kernel of the matrix describing the coupled constraints. However, in the large-scale context and for general linear constraints it is very difficult to generate such sketch matrices. Another strand of this literature develops  and analysis center-free gradient methods \cite{XiaBoy:06}, augmented Lagrangian based methods \cite{HonLuo:12} or Newton methods \cite{WeiOzd:13}.

\vspace{0.1cm}

\textbf{Our approach and contribution}.   Our approach introduces general sketch descent algorithms for solving large-scale smooth  optimization problems with multiple linear coupled constraints. Since we have non-separable constraints in the problem formulation, a random sketch descent scheme needs to consider new sampling rules for choosing the coordinates. We first investigate  conditions on the sketching of the coordinates over which we minimize at each iteration in order  to have well-defined algorithms. Based on these conditions  we develop new random sketch descent methods for solving our linearly constrained convex problem, in particular, random sketch descent and accelerated  random sketch descent methods.  However, unlike existing  methods such as coordinate descent, our algorithms are capable of utilizing curvature information, which leads to striking improvements in both theory and practice.

\noindent \textit{Our contribution.} To this end, our main contribution can be summarized as follows:

(\textit{i}) Since we deal with  optimization problems having non-separable constraints we need to design sketch descent schemes based on new sampling rules for choosing the sketch matrix. We derive necessary and sufficient conditions on the sketching of the coordinates over which we minimize at each iteration in order  to have well-defined algorithms. To our knowledge, this is the first complete work on random sketch descent type algorithms for problems with more than one linear constraint. Our theoretical results consist of new optimization algorithms, accompanied with global convergence guarantees to solve a wide class of non-separable optimization problems.

(\textit{ii}) In particular, we propose a random sketch descent algorithm for solving such general  optimization problems. For the general case, when the objective function is smooth and non-convex, we prove sublinear rate in expectation for an appropriate optimality measure. In the smooth convex  case we obtain in expectation an $\epsilon$-accurate solution in at most ${\cal O}(1/\epsilon)$ iterations, while for strongly convex functions the method converges linearly.

(\textit{iii}) We also propose an accelerated random sketch descent algorithm. From our knowledge, this is the first analysis of an accelerated variant for optimization problems with non-separable linear constraints. In the smooth convex  case we obtain in expectation an $\epsilon$-accurate solution in at most  ${\cal O}(1/\sqrt{\epsilon})$ iterations. For strongly convex functions the new random sketch descent method converges linearly.

\noindent Let us emphasize the following points of our contribution. First, our sampling strategies are for multiple linear  constraints and thus very different from the existing methods designed only for one linear constraint. Second, our (accelerated) sketch descent schemes are the first  designed for this class of problems. Thirdly, our non-accelerated sketch descent algorithm covers as special cases some methods designed for problems with a single linear constraint and coordinate sketch. In these special cases, where convergence bounds are known, our theory recovers the best known bounds. We also illustrate,  that for some problems, random sketching  of the coordinates produces better results than  deterministic selection of them. Finally, our theory can be used to further develop other methods such as Newton-type schemes.

\vspace{0.1cm}

\textbf{Paper organization}.  The rest of this paper is organized as follows. Section 2 presents necessary and sufficient conditions for the sampling of the sketch matrix. Sections 3  provides a full convergence analysis of the random sketch descent method, while Section 4 extends this convergence analysis to an accelerated variant. In Section 5 we show  the benefits of general sketching over  fixed selection of coordinates.


\subsection{Problem formulation}
\noindent We consider the following large-scale general smooth  optimization problem with multiple linear coupled constraints:
\begin{align}
\label{problem}
& f^* = \min_{x \in \cR^n} f(x)  \quad  \text{s.t.} \quad  Ax=b,
\end{align}
where $f: \R^n \to \R$ is a general differentiable function and $A \in \R^{m \times n}$, with $m \ll n$,  is such that the feasible set is nonempty. The last condition is satisfied if e.g. $A$ has full row rank. The simplest case is when $m=1$, that is  we have a single linear constraint $a^Tx=b$ as considered in \cite{Bec:12,Nec:13,NecPat:14,NecNes:11}.  Note that we do not necessarily impose $f$ to be a convex function.   From the optimality conditions for our  optimization  problem \eqref{problem} we have that $x^* \in \R^n$ is a stationary point if there exists some $\lambda^* \in \R^m$ such that:
\[ \nabla f(x^*) + A^T \lambda^* =0 \quad \text{and} \quad A x^*=b. \]
However, if $f$ is convex, then any $x^*$ satisfying the previous optimality conditions is a global optimum for optimization problem \eqref{problem}. Let us define $X^*$ the set of these points. Therefore, $x^* \in X^*$ is a stationary (optimal) point if it is feasible and satisfies the condition:
\[ \nabla f(x^*) \in \range(A^T). \]


\subsection{Motivation}
We  present below several  important applications
from which the interest for problems of type \eqref{problem} stems.

\subsubsection{Page ranking}
\label{sec:pr}
This problem has many applications in google ranking, network control, data analysis \cite{IshTem:12,Nes:10,Nec:13}. For a given graph ${\cal G}$ let   $\bar E \in \R^{n \times n}$ be its incidence matrix, which is sparse. Define $E = \bar E \ \text{diag}(\bar E^T e)^{-1}$, where $e \in \R^n$ denotes
the vector with all entries equal to $1$. Since $E^T e = e$, i.e.
the matrix $E$ is column stochastic, the goal is to determine a
vector $x^*$ such that:  $E x^* = x^*$  and $e^T x^*  =1$. This problem can be written directly in  optimization form:
\begin{equation*}
\min_{x \in \R^n} \; f(x) \quad \left(:= \frac{1}{2} \norm{Ex - x}^2 \right) \quad \text{s.t.} \quad e^T x  = 1,
\end{equation*}
which is a particular case of our optimization problem  \eqref{problem} with $m=1$ and  $E$ sparse matrix.

\subsubsection{Machine learning}
\label{sec:ml}
Consider the optimization problem associated with  the loss minimization of linear predictors without regularization for a training data set  containing $n$ observations $a_i \in \R^m$ \cite{ShaZha:13}:
\[ \min_{w \in \R^m} \frac{1}{n} \sum_{i=1}^n \phi_i(w^T a_i).  \]
Here $\phi_i$ is some loss function, e.g.  SVM $\phi_i(z) = \max \{0, 1 - y_i z\}$, logistic regression $\phi_i(z) = \log(1 + \exp(-y_iz))$, ridge regression $\phi_i(z) = (z - y_i)^2$, regression with the absolute value $\phi_i(z) = |z- y_i|$ and support vector regression $\phi_i(z) = \max \{0, |z - y_i| - v\}$ for some predefined insensitivity parameter $v>0$. Moreover, in classification the labels $y_i \in \{-1, 1\}$, while in regression $y_i \in \R$.   Further, let $\phi_i^*$ denote  the Fenchel conjugate  of $\phi_i$. Then the dual of this problem becomes:
\[  \min_{x \in \R^n} f(x) \quad \left(=\frac{1}{n} \sum_{i=1}^n \phi_i^*(x_i)\right) \quad \text{s.t.} \quad Ax=0, \]
where $A=[a_1 \cdots a_n] \in \R^{m \times n}$. Clearly, this problem fits into our   model \eqref{problem}, with $m$ representing the number of features, $n$ the number of training data, and the objective function $f$ is separable.


\subsubsection{Portfolio optimization}
\label{sec:po}
In the basic Markowitz portfolio selection model \cite{Mar:52},  see also \cite{FroRei:15} for related formulations,  one assumes a set of $n$ assets, each with expected returns $\mu_i$, and a covariance matrix
$\Sigma \in \R^{n \times n}$, where $\Sigma_{(i,j)}$ is the covariance between returns of assets $i$ and $j$. The goal is to allocate a portion of the budget into different assets, i.e.  $x_i \in \R$  represents a portion of the wealth to be invested into asset $i$, leading to the first constraint: $\sum_{i=1}^n x_i = 1$.  Then, the expected return (profit) is $r  = \sum_{i=1}^n \mu_i x_i$ and the variance of the portfolio can be computed as $\sum_{i,j}  x_i x_j \Sigma_{(i,j)}$. The investor seeks to minimize risk (variance) and maximize the expected return, which is usually formulates as maximizing profit while limiting the risk or minimizing risk while requiring given expected return. The later formulation can be written as:
\begin{align*}
\min_{x \in \R^n}  x^T \Sigma x \quad \mbox{s.t.} \quad  \sum_{i=1}^n  \mu_i x_i = r, \;\; \sum_{i=1}^n x_i = 1,
\end{align*}
which clearly fits again into our optimization model   \eqref{problem} with $m=2$. We can further  assume that each asset belongs exactly to one class $c \in [C]$, e.g. financials, health care, industrials, etc. The investor would like to diversify its portfolio in such a way that the net allocation in class $c$ is $a_c$: $\sum_{i=1}^n x_i \textbf{1}_c(i) = a_c$ for all $c \in [C]$, where $\textbf{1}_c(i) = 1$ if asset $i$ is in class $c$ and $\textbf{1}_c(i) = 0$ otherwise. One can observer that in this case we get a similar problem as above, but with  $C$ additional linear constraints ($m=C+2$).


\section{Random sketching}
It is important to note that stochasticity enters in our algorithmic framework  through a user-defined distribution ${\sS}$ describing an ensemble of random matrices $S \in  \R^{n \times p}$ (also called \textit{sketch} matrices). We assume that $p \ll n$, in fact we usually require $p \sim {\cal O}(m)$ and note that $p$ can also be random (i.e. the ${\sS}$ can return matrices with different $p$).  Our schemes  and the underlying convergence theory support virtually all thinkable distributions. The choice of the distribution should ideally depend on the problem itself, as it will affect the convergence speed. However, for now we leave such considerations aside. The basic idea of our algorithmic framework consists of a given  feasible $x$,  a sample sketch matrix $S \sim {\sS}$ and a basic update of the form:
\begin{equation}
\label{eq:updateRuleIntro}
 x^+  = x + S d \quad \mbox{such that} \quad ASd = 0,
\end{equation}
where the requirement $A S d= 0$ ensures that the new point $x^+$ will remain feasible.  Clearly, one can choose a distribution ${\sS}$ which will not guarantee convergence to stationary/optimal point. Therefore, we need to impose some minimal necessary conditions for such a scheme to be well-defined. In particular, in order to avoid trivial updates, we need to choose $S \sim {\sS}$ such that the homogeneous  linear system $A S d=0$ admits also nontrivial solutions, that is we require:
\begin{align}
\label{eq:suf1}
\range(S) \cap \ker(A) \neq  0.
\end{align}
Moreover, since for any feasible $x^0$ an optimal solution satisfies $ x^* \in x^0 + \ker(A)$, it is necessary to require that with our distribution $\sS$ we can generate $\ker(A)$:
\begin{align}
\label{eq:suf2}
\ker(A) = \text{Span}\left(\cup_{S \sim {\sS}} \left(\range(S) \cap \ker(A)\right)\right).
\end{align}
Note that the geometric conditions \eqref{eq:suf1}-\eqref{eq:suf2} are only necessary  for a sketch  descent type scheme to be well-defined. However,  for a discrete probability distribution, having e.g. the property that $\Prob(S) > 0$ for all $S \sim {\sS}$, condition \eqref{eq:suf2} is also sufficient. In Section \ref{sect:prelim} (see Assumption \ref{ass_Z}) we will provide sufficient conditions for a general probability distribution ${\sS}$ in order to obtain well-defined algorithms based on such sketching. Below we provide several examples of distributions satisfying our geometric conditions \eqref{eq:suf1}-\eqref{eq:suf2}.



\subsection{Example 1 (finite case)} Let us consider a finite (or even countable) probability distribution ${\sS}$.  Further, let $x^0$ be a particular solution of the linear system $Ax=b$. For example, if $A^\dagger$ denotes the pseudo-inverse of the  matrix $A$, then we can take $x^0 = A^\dagger b$. Moreover, by the properties of the pseudo-inverse,  $I_n - A^\dagger A$ is a projection matrix onto $\ker(A)$, that is $\range(I_n - A^\dagger A) = \ker(A)$. Therefore, any solution of the linear system $Ax=b$ can be written  as:  \[ x= A^\dagger b + (I_n - A^\dagger A) y, \]
for any $y \in \R^n$. Thus, we may consider a finite (the extension to countable case is straightforward) set of matrices $\Omega = \{S_i \in \R^{n \times p}: \; i=1:N \}$ endowed with a probability distribution $P_i=\Prob(S=S_i)$ for all $i \in [N]$ and condition \eqref{eq:suf2} requires that the span of the image spaces of $\{S_i\}_{i=1}^N$ contains or is equal to $\range(I_n - A^\dagger A)$:
\begin{equation}
\label{eq:suf22}
\ker(A) = \range(I_n - A^\dagger A)  = \text{Span}\left({\displaystyle \cup}_{i: P_i >0} (\range(S_i) \cap \ker(A))\right).
\end{equation}
In particular, we have several choices for the sampling for a finite distribution:
\begin{enumerate}[nosep]
\item If one can compute  a basis for $\ker(A)$, then we can take as random sketch matrix $S_i \in \R^{n \times p}$ any  block of $p$ elements  of this basis endowed  with some probability $P_i=\Prob(S=S_i) >0$ (for the case $p=1$ the matrix  $S_i$ represents a single element of this basis generating $\ker(A)$). This sampling was also considered in \cite{FroRei:15}.  Clearly, in this particular case condition  \eqref{eq:suf1} and condition \eqref{eq:suf2} or equivalently  \eqref{eq:suf22} hold since  $\ker(A) = \text{Span}\left(\cup_{i=1}^N \range(S_i)\right)$.

\item However, for a general matrix $A$ it is difficult to compute a basis of $\ker(A)$. A simple alternative is to consider then any $p$-tuple ${\cal N} = (i_1 \cdots i_p) \in 2^{[n]}$, with $p > m$,  and the corresponding random sketch matrix $S_{\cal N} = [e_{i_1} \cdots e_{i_p}]$, where $e_i$ denotes the $i$th column of the identity matrix $I_n$, with some probability distribution $P_{\cal N}$ over the set of $p$-tuples in $2^{[n]}$.  It is clear that for this choice condition \eqref{eq:suf1}  and condition \eqref{eq:suf2} or equivalently \eqref{eq:suf22} also hold. For the particular case when we have a single linear coupled constraint, i.e. $a^T x=b$, we can take random  matrices $S_{(ij)}=[e_i \ e_j]$ also considered e.g.  in \cite{Nec:13}. This particular sketch matrix based on sampling columns  of the identity matrix leads to coordinate descent framework. However, the other examples (including those from Section \ref{sect:ex_inf}) show that our sketching framework is more general than coordinate descent.

\item Instead of working with the matrix $I_n$, as considered previously,  we can  take any orthogonal or full rank matrix ${\cal I} \in \R^{n \times n}$ having the columns ${\cal I}_{i}$ and thus forming a basis of $ \R^{n}$. Then, we can consider  $p$ tuples ${\cal N} =(i_1,\cdots,i_p) \in 2^{[n]}$, with $p>m$, and the corresponding random sketch  matrix  $S_{\cal N} =[{\cal I}_{i_1} \cdots {\cal I}_{i_p}]$, with some probability distribution $P_{\cal N}$ over the set of $p$-tuples in $2^{[n]}$. It is clear that for this choice of the random sketch matrices $S$ the  condition \eqref{eq:suf1} and condition \eqref{eq:suf2} or equivalently \eqref{eq:suf22} still hold.
\end{enumerate}


\subsection{Example 2 (infinite case)}
\label{sect:ex_inf}
Let us now consider a continuous (uncountable)  probability distribution ${\sS}$.
We can consider in this case two simple sampling strategies:
\begin{enumerate}[nosep]
\item If one can sample easily a random matrix $B$ such
that $\range(B) = \ker(A)$, then we can choose one or several columns from this matrix as a sketch matrix $S$. In this case $p \geq 1$.

\item Alternatively, we can sample random full rank matrices in  $\R^{n \times n}$
and then define $S$ to be random $p>m$ columns. Furthermore, since it is known that random Gaussian matrices are full rank almost surely, then we can define $S \sim \mathcal{N}^{n\times p}$ to be a random Gaussian matrix. Similarly, we can consider random uniform matrices and define e.g. $S \sim \mbox{Unif}(-1, 1)^{n \times p}$.
\end{enumerate}

\noindent A sufficient condition for a well-defined  sampling in  the infinite case  is to ensure that in expectation one can move in any direction in $\ker(A)$.
Considering the general update rule \eqref{eq:updateRuleIntro}, we see that if we sample $S\in \R^{n\times p}$, then our update can be only $S d$ for some $d \in \R^p$.
Now, we also have a condition, that we want to stay in the $\ker(A)$, and therefore $d$ cannot be anything, but has to be chosen such that $S d \in \ker(A)$. Now, this restricts the set of possible $d$'s to be such that:
$$
ASd = 0 \qquad \Rightarrow \qquad
d = (I_p - (AS)^\dagger (AS) ) t$$
for some $t \in \R^p$. Recall, that we allow $p$ to be also random, hence to derive the sufficient condition we need to have some quantity with dimension independent on $p$.
Note that  each $t \in \R^p$ can be represented as $S^T t'$ for some (possibly non-unique) $t'$. Therefore, we see that if $S$ is sampled, then we can move in the direction:
$$ Sd = S (I - (AS)^\dagger (AS) ) t = S (I - (AS)^\dagger (AS) ) S^T t', $$
hence, we have the ability to move in $\range \big( S (I - (AS)^\dagger (AS) ) S^T \big )$. Now, the  condition to be able to move in $\ker(A)$ can be expressed as
requiring that on expectation we can move anywhere in $\ker(A)$:
\begin{equation}
\label{eq:SNcondition}
\range \left( \Exp \left[ 
{ S  (I - (AS)^\dagger(AS) ) S^T}
\right]  \right) = \ker(A),
\end{equation}
provided that  the expectation exists and is finite. Note, that
this condition  must hold also  for a discrete probability distribution, however  the condition \eqref{eq:suf2}  is  more intuitive in the discrete case. In the next section we provide algebraic sufficient conditions on the sampling for a general  probability distribution ${\sS}$ in order to obtain well-defined algorithms.


\subsection{Sufficient conditions for sketching}
\label{sect:prelim}
It is well known that in order to derive any reasonable convergence guarantees for a minimization scheme we need to impose some smoothness property on the objective function.  Therefore, throughout the paper  we consider the following blanket assumption on the smoothness of $f$:
\begin{assumption}
\label{ass_lip}
For any feasible $x^0$ there exists a positive semidefinite matrix $M$ such that $M$ is positive definite on $\ker(A)$ and the following inequality holds:
\begin{equation}
\label{eq:M}
f(y) \leq f(x) + \ve{\nabla f(x)}{y-x} + \frac12 (y-x)^T M (y-x), \quad \forall x,y \in x^0 + \ker(A).
\end{equation}
\end{assumption}

%

\noindent Note that for a general (possibly non-convex) differentiable function $f$ the smoothness  inequality  \eqref{eq:M} does not imply  that the objective function  $f$  has Lipschitz continuous gradient, so our assumption is less conservative than requiring Lipchitz gradient assumption. However, when $f$ is convex the condition \eqref{eq:M} is equivalent  with Lipschitz continuity of the gradient of $f$ on $x^0 + \ker(A)$ \cite{Nes:04}.  In particular,  if $M=L \cdot I_n$ for some Lipschitz constant $L>0$ we recover the usual definition of Lipschitz  continuity of the gradient for the class of convex functions.  Our sketching methods derived below are based on \eqref{eq:M}  and therefore they have the capacity to utilize curvature information.  In particular, if the objective function is quadratic, our methods can  be  interpreted  as    novel  extensions to more general optimization models  of  the  recently introduced iterative Hessian sketch method for minimizing self-concordant objective functions \cite{pilanci2017newton}. The reader should also note that we can further relax  the condition \eqref{eq:M} and require smoothness of $f$ with respect to any image space generated by the random matrix $S$. More precisely, it is sufficient to assume that for any sample  $S \sim {\sS}$ there exists a positive semidefinite matrix $M_S$ such that $M_S$ is positive definite on $\ker(A)\cap \range(S)$ and the following inequality holds:
\begin{equation*}
f(y) \leq f(x) + \ve{\nabla f(x)}{y-x} + \frac12 (y-x)^T M_S (y-x) \quad \forall x,y \in x^0+\ker(A) \ \wedge  \ x-y \in \range(S).
\end{equation*}
Note that if $M_S = M$  for all $S$ we recover the relation \eqref{eq:M}. For simplicity of the exposition in the sequel we assume \eqref{eq:M} to be valid, although all our convergence results can be also extended under previous smoothness condition given in terms of $M_S$.

From the above discussion  it is clear that  the direction $d$ in our basic update \eqref{eq:updateRuleIntro}  needs to be in the kernel of matrix $AS$. However, it is well known that the projection onto $\ker(A S)$ is given by the projection matrix:
\[ P_S=  I_p - (AS)^\dagger (AS). \]
Clearly, we have $\ker(AS) = \range(P_S)$. Let us further define the matrix:
\begin{equation}
\label{eq:defZT}
Z_S = S P_S  (P^T_S  S^T M S P_S)^\dagger P^T_S S^T \in \R^{n \times n}.
\end{equation}
The matrix $Z_S$ has some important properties that we will derive below since they are useful for algorithm development. First we observe that:
\begin{lemma}
\label{lema1}
For any probability distribution ${\sS}$ the matrix $Z_S$ is symmetric ($Z_S=Z_S^T$), positive semidefinite ($Z_S \succeq 0$), and  for any $u \in \range(A^T)$ we have $ Z_S u = 0$, that is $\range(A^T) \subseteq \ker(Z_S)$. Moreover, the following identity holds  $Z_S M Z_S= Z_S$.
\end{lemma}

\begin{proof}
It is clear that $Z_S$ is positive semidefinite matrix since $M$ is assumed positive semidefinite. It is well-known that for any given matrix $B$ its pseudo-inverse satisfies $B B^\dagger B= B$ and $B^\dagger B B^\dagger = B^\dagger$. Now, for the first statement given the expression of $Z_S$ it is sufficient to prove that $P^T_S  S^T u = 0$ for $u \in \range(A^T)$. However, if $u \in \range(A^T)$ then there exists $y$ such that $u = A^Ty$ and consequently
we have:
\begin{align*}
P^T_S  S^T u & = P^T_S  S^T  A^T y = (I - (AS)^\dagger (AS))^T ( A S)^T y \\
&= [(AS) (I - (AS)^\dagger (AS))]^Ty \\
& = (( A S) - (AS) (AS)^\dagger ( A S))^T y =0,
\end{align*}
where in the last equality we used the first property of  pseudo-inverse
$(A S) (A S)^\dagger (A S) = A S$. For the second part of the lemma we use the expression of $Z_S$ and the second property of the pseudo-inverse applied to the matrix   $ (P^T_S  S^T M S P_S)^\dagger$, that is:
\[ Z_S M Z_S =  [S P_S  (P^T_S  S^T M S P_S)^\dagger P^T_S S^T] M [S P_S  (P^T_S  S^T M S P_S)^\dagger P^T_S S^T] = Z_S, \]
which concludes our statements.
\end{proof}

\noindent Now, since the random matrix $Z_S$ is positive semidefinite, then we can define its expected value, which is also a symmetric positive semidefinite matrix:
\begin{equation}
\label{eq:def_of_Z}
Z = \E_S [Z_S].
\end{equation}
In the sequel we also consider the following assumption on the expectation matrix $Z$:
\begin{assumption}
\label{ass_Z}
We  assume that the distribution ${\sS}$ is chosen such  that $Z_S$ has a finite mean, that is the matrix  $Z$ is well defined, and positive definite (notation $Z \succ 0$) on $\ker(A)$.
\end{assumption}
As we will see below, Assumption \ref{ass_Z} is a sufficient condition on the probability distribution ${\sS}$ in order to ensure convergence of our algorithms that will be defined in the sequel. To our knowledge this algebraic characterization of the probability distribution defining  the sketch matrices $S$ for problems with multiple non-separable linear constraints seems to be new.

Note that the necessary condition \eqref{eq:suf1} holds provided that $Z_S \not = 0$.  Indeed, from Lemma \eqref{lema1}  we have $\range(A^T) \subseteq \ker(Z_S)$ for all $S \sim \sS$ and $\ker(Z_S) \perp \range(Z_S)$. Therefore, we get that $\range(A^T) \perp \range (Z_S)$ and we know that  $\range(A^T) \perp \ker(A)$. Let $z \in \range (Z_S) \subseteq \R^n$, $z \not = 0$, then there exists unique $z_1 \in \range(A^T)$ and $z_2 \in \ker(A)$ such that $z=z_1 + z_2$. Moreover, we have $z \perp \range(A^T)$, i.e. $z \perp z_1$, which implies that $\langle z_1+z_2, z_1\rangle = \|z_1\|^2 + 0 =0$. Thus, $z_1 = 0$ and $z \in \ker(A)$. From the last  relation, we get:
\[  \range(Z_S) \subseteq \ker(A). \]
Moreover, from the definition of the symmetric matrix $Z_S$ we have $\range(Z_S) \subseteq \range(S)$, which combined with the previous relation leads to:
\[  \range(Z_S) \subseteq \ker(A) \cap \range(S),  \]
and consequently proving that the condition  \eqref{eq:suf1} holds provided that $Z_S \not = 0$. Moreover, we can show  that the necessary condition  \eqref{eq:suf2} holds if  $Z$ satisfies Assumption \ref{ass_Z}:

\begin{lemma}
\label{lem3}
Under  Assumption \ref{ass_Z}  the necessary condition \eqref{eq:suf2} is valid. Additionally, the following identity takes place:
\[  \range(A^T) = \ker(Z) \]
and consequently $Z^\dagger Z$ is a  projection matrix onto $\ker(A)$, where $Z^\dagger$ denotes the pseudo-inverse of  the matrix $Z$.
\end{lemma}

\begin{proof}
Note that Assumption \ref{ass_Z} holds, i.e. $Z \succ 0$ on $\ker(A)$, if and only if  $Z \succ 0$ on  $\R^n \setminus \range(A^T)$. Moreover, for any non-zero $u \in \ker(A)$, we have $Z u \not =0$, that is $u \not \in \ker(Z)$. In conclusion, we get $\ker(A) \subseteq \R^n \setminus \ker(Z)$. But, $\range(Z) \subseteq \text{Span}(\cup_{S \sim {\sS}} \range(Z_S))$, from which we can conclude \eqref{eq:suf2}.

For the second part we use again Lemma \eqref{lema1}:  $\range(A^T) \subseteq \ker(Z_S)$ for all $S \sim {\sS}$. This means that $ \range(A^T) \subseteq \cap_{S \sim {\sS}}\ker(Z_S) \subseteq \ker(Z)$. The other inclusion follows by reducing to absurd. Assume that there exists $u \not \in \range(A^T)$ such that $Zu=0$, or equivalently $Zu=0$ for some $u \in \R^n \setminus \range(A^T)$. However, note that   $Z \succ 0$ on $\ker(A)$ if and only if  $Z \succ 0$ on  $\R^n \setminus \range(A^T)$, which contradicts our assumption. In conclusion, the second statement holds. Finally, it is well-known that $Z^\dagger Z$ is an orthogonal projector onto $\range(Z^T)$ and the rest follows from standard algebraic arguments.
\end{proof}

\noindent {\bf The primal-dual "norms"}.  Since the matrix $Z_S$ is positive semidefinite, matrix $Z$ is also positive semidefinite. Moreover, from Lemma \ref{lema1} we conclude that  $\range(A^T) \subseteq \ker(Z)$. In the sequel we assume that $S \sim {\sS}$ such that $Z$ is a positive definite matrix on  $\ker(A)$ and consequently on $\R^n \setminus \range(A^T)$ (see Assumption \ref{ass_Z}). Then, we can define  a norm induced by the matrix $Z$ on $\ker(A)$ or even
$\R^n \setminus \range(A^T)$. This  norm  will be used subsequently for measuring distances in the  subspace $\ker(A)$. More precisely, we define the \textit{primal norm} induced by the positive semidefinite matrix $Z$ as:
\[ \| u \|_{Z} = \sqrt{u^T {Z} u} \quad \forall u \in \R^{n}. \]
Note that $\| u \|_{Z} = 0$ for all $u \in \range(A^T)$ (see Lemma \ref{lema1}) and
$\| u \|_{Z} > 0$ for all $u \in \R^n \setminus \range(A^T)$. On the subspace $\ker(A)$ we introduce the extended dual norm:
\begin{align*}
 \| x \|_{Z}^* = \max_{u \in \R^{n}: \|u \|_{Z} \leq 1} \langle x,
u \rangle  \quad \forall x \in \ker(A).
\end{align*}

\noindent Using the definition of conjugate norms, the
Cauchy-Schwartz inequality holds:
\begin{align}
\label{eq:CS}
\langle u, x \rangle \leq \| u \|_{Z} \cdot \|
x \|_{Z}^* \quad \forall x \in \ker(A), \;  u \in \R^{N}.
\end{align}

\begin{lemma}
\label{lem4}
Under Assumption \ref{ass_Z} the primal and dual norms have the following expressions:
\begin{align}
\label{norm_pseudo}
& \| u \|_{Z} = \sqrt{u^T {Z} u}, \quad  \| x \|_{Z}^* = \sqrt{x^T Z^\dagger x} \quad \forall u \in \R^{n}, \quad \forall x \in \ker(A).
\end{align}
\end{lemma}

\begin{proof}
\noindent Let us consider any   $\hat u \in \range(A^T)$.  Then, the dual norm can be computed
for any $x \in \ker(A)$ as follows:
\begin{align*}
& \| x\|_{Z}^*  = \max_{u \in \R^{n}: \; \langle {Z} u, u \rangle \leq 1} \langle x, u \rangle = \max_{u: \langle {Z} \left( u -  \hat u \right), u - \hat u  \rangle \leq 1} \langle x, u - \hat u  \rangle  \\
& = \max_{u: \langle {Z} u, u \rangle \leq 1, u \in \ker(A)}
\langle x, u \rangle = \max_{u: \langle {Z} u, u \rangle \leq 1, Au=0 } \langle x,
u \rangle \\
&  = \max_{u: \langle {Z} u, u \rangle \leq 1, u^T A^T A u \leq
0 } \langle x, u \rangle \\
& = \min_{\nu, \mu \geq 0} \max_{u \in \R^n}  [\langle x, u \rangle + \mu(1 -
\langle {Z} u, u \rangle) - \nu \langle  A^T A u, u \rangle]  \\
& =  \min_{\nu, \mu \geq 0} \mu + \langle (\mu {Z} + \nu A^T A)^{-1} x, x \rangle  =
\min_{\nu \geq 0} \min_{\mu \geq 0} [\mu + \frac{1}{\mu} \langle
({Z} + \frac{\nu}{\mu} A^T A)^{-1} x, x \rangle]  \\
& = \min_{\zeta \geq 0} \sqrt{\langle ( {Z} + \zeta A^T A)^{-1} x,
x \rangle}.
\end{align*}

\noindent We obtain an extended   dual norm  that is
well defined  on  the subspace $\ker(A)$:
\begin{align}
\label{primalnorm}
 \| x \|_{Z}^* = \min_{\zeta \geq 0} \sqrt{\langle \left(
{Z} + \zeta A^T A \right)^{-1}  x, x \rangle} \quad
\forall x \in \ker(A).
\end{align}

\noindent The eigenvalue decomposition  of the positive
semidefinite matrix $Z$ can be written as ${Z} = U \text{diag}(\lambda_1, \cdots, \lambda_r, 0, \cdots, 0) U^T$, where $\lambda_i$ are its
positive eigenvalues and the columns of orthogonal matrix $U=[U_{ker} \; U_{range}]$ are the corresponding eigenvectors, $U_{ker}$ generating $\ker(A)$ and $U_{range}$ generating $\range(A^T)$. Then, we have:
\[ ({Z} + \zeta A^T A)^{-1} = U \text{diag}(\lambda_1, \cdots, \lambda_r, \zeta \lambda_{r+1}, \cdots, \zeta \lambda_n)^{-1} U^T, \]
where $\lambda_{r+1}, \cdots, \lambda_n$ are the nonzero eigenvalues of symmetric matrix $A^TA$. From \eqref{primalnorm} it follows  that our newly defined dual norm has the following closed form:
\begin{align*}
\| x \|_{Z}^* = \sqrt{x^T Z^\dagger  x} \quad \forall x \in \ker(A),
\end{align*}
where $Z^\dagger$ denotes the pseudoinverse of  matrix $Z$.
\end{proof}

The following example shows that the 2-coordinate sampling proposed in  \cite{NecNes:11}
(in the presence of  a single linear constraint $m=1$) is just a special case of the sketching analyzed in this paper:
\begin{example}
\label{example}
Let us consider the following  optimization problem:
\[ f(x) = \sum_{i=1}^n f_i(x_i) \quad \mbox{subject to} \quad  \sum_{i=1}^n x_i = b.  \]
In this case, assuming that each scalar function $f_i$ has $L_i$ Lipschitz continuous gradient, then $M = \text{diag}(L_1, \cdots, L_n)$. Moreover, we can take any random pair of coordinates $(i,j)$ with $i,j=1:n, \; i<j$ and consider the particular sketch matrix $S_{(ij)} = [e_i \; e_j]$. Note that, for simplicity, we focus here on Lipschitz dependent  probabilities for choosing the pair $(i,j)$, that is $P_{(i,j)} = (L_i+L_j)/(n-1)L$ with $L = \sum_{i=1}^n L_i$. Following basic derivations  we get:
\begin{align}
\label{eq:Zij}
& Z_{(ij)} = \frac{1}{L_i+L_j} S_{(ij)} \begin{bmatrix} 1 & -1\\ -1 & 1\end{bmatrix} S_{(ij)}^T = \frac{1}{L_i+L_j} (e_i - e_j) (e_i - e_j)^T, \nonumber \\
& Z= \frac{n}{(n-1)L} \left(I_n - \frac{1}{n} ee^T\right),  \quad    \quad  Z^\dagger= \frac{(n-1)L}{n} \left(I_n - \frac{1}{n} ee^T\right).
\end{align}
Clearly, $Z \succ 0$ on $\ker(A)$ and thus Assumption \ref{ass_Z} holds. Similarly, we can compute explicitly $Z$ and $Z^\dagger$ for the fixed selection of the pair of coordinates $(i,i+1)$ with $i=1:n-1$.
\end{example}



\section{Random Sketch Descent (\RSD\!\!) algorithm}
For the large-scale optimization problem \eqref{problem} methods which scale cubically, or even quadratically, with the problem size $n$  is already out of the question; instead, linear scaling of the computational costs per-iteration is desired. Clearly, optimization problem \eqref{problem} can be solved using projected  first order methods, such as gradient or accelerated gradient, both algorithms having comparable cost per iteration \cite{Nes:04}. In particular, both methods require the computation of the full gradient $\nabla f(x)$ and finding  the optimal solution of a subproblem with quadratic objective over the subspace $\  ker(A) \subset \R^n$:
\begin{align}
\label{eq:gsubproblem}
\min_{d \in \R^n: A d=0}  f(x) + \ve{\nabla f(x)}{d} + \frac12 \ d^T M  d.
\end{align}
For example, for the projected gradient method since we assume $M$  positive definite on $\ker(A)$ (see Assumption \ref{ass_lip}), then the previous subproblem has a unique solution leading to  the following gradient iteration:
\begin{align}
\label{eq:gupdate}
x^{k+1}_G = x^k_G - Z_{I_n} \nabla f(x^k_G),
\end{align}
where $Z_{I_n} \in \R^{n \times n}$ is obtained by replacing $S= I_n$ in the definition of the matrix $Z_S$. However, for very large $n$ even the first iteration is not computable, since the cost of computing $Z_{I_n}$ is cubic in the problem dimension (i.e. of order ${\mathcal{O}(n^3)}$ operations) for a dense matrix $M$.   Moreover, since usually $Z_{I_n}$ is a dense matrix regardless of the matrix $M$ being dense or spare, the cost of the subsequent iterations is quadratic  in the problem size $n$ (i.e. ${\mathcal{O}(n^2)}$). Therefore, the development of new optimization algorithms that target linear cost per iteration and nearly dimension-independent convergence rate is needed. These properties can be achieved using the sketch descent framework. In particular, let us assume that the initial iterate $x^0$ is a feasible point, i.e. $Ax^0 = b$. Then, the first algorithm we propose, Random Sketch Descent (\RSD\!\!) algorithm,  chooses at each iteration a random sketch matrix $S \in \R^{n \times p}$ according to the probability distribution ${\sS}$ and  find a new direction solving a  simple subproblem (see  Algorithm~\ref{alg:rcd} below).
\begin{algorithm}[h!]
\caption{Algorithm \RSD}
\label{alg:rcd}
\begin{algorithmic}[1]
\STATE choose $x^0 \in \R^n$ such that $Ax^0 = b$
\FOR { $k \geq 0$ }
\STATE Sample $S \sim {\sS}$ and perform the update:
\STATE  $x^{k+1} = x^k -  Z_S \nabla f(x^k)$.
\ENDFOR
\end{algorithmic}
\end{algorithm}
\noindent Let us explain the update rule of our algorithm \RSD\!. Note that the new direction in the update $x^{k+1}= x^k + S d^k$  of \RSD  is computed from a subproblem with quadratic objective over the subspace  $\ker(AS) \subset \R^p$ that it is simpler than subproblem  \eqref{eq:gsubproblem} corresponding to the full gradient:
\[ d^k =\displaystyle  \arg \min_{d \in \R^p: A S d=0}  f(x^k) + \ve{\nabla f(x^k)}{Sd} + \frac12 \ d^T  S^T M S d. \]
We observe that from the feasibility condition $A S d = 0$ we can compute $d$ as:
\[ d = P_S t \quad \left(:= (I_p - (AS)^\dagger (AS)) t\right), \]
for some $t$. Then, the constrains will not be violated. Now, let's plug this into the objective function of the subproblem, to obtain an unconstrained problem in $t$:
\[  t^k = \arg \min_{t \in \R^p} \ve{\nabla f(x^k)}{S ((I_p - (AS)^\dagger (AS)) t )} +\frac12 \| S(I_p - (AS)^\dagger (AS)) t \|_M^2. \]
Then, from the first order optimality conditions we obtain that:
\[  P_S^T S^T M S P_S t^k = - P_S^T S^T \nabla f(x^k), \]
and hence we can define $t^k$ as
\[  t^k = -(P_S^T S^T M S P_S)^\dagger P_S^T S^T \nabla f(x^k). \]
In conclusion we obtain the following update rule for our \RSD algorithm:
\begin{equation}
\label{eq:rcdupdate}
x^{k+1} = x^k - \underbrace{S P_S (P_S^T S^T M S P_S)^\dagger P_S^T S^T}_{=Z_S} \nabla f(x^k) = x^k - Z_S \nabla f(x^k).
\end{equation}
After $k$ iterations of the \RSD algorithm, we generate a random output $(x^k, f(x^k))$, which depends on the observed implementation of the
random variable:
\[ {\cal F}_k =  (S_0, \cdots,  S_{k-1}). \]
Let us define the expected value of the objective function w.r.t. ${\cal F}_k$:
\[  \phi_k = \E \left[ f(x^k) \right]. \]
Next, we compute the decrease of the objective function after one random step:
\begin{align*}
f(x^{k+1}) & = f(x^k + S P_S  t^k) = f(x^k - Z_S  \nabla f(x^k))\\
& \overset{\eqref{eq:M}}{\leq} f(x^k) - \ve{\nabla f(x^k)}{Z_S \nabla f(x^k)} + \frac12 \|Z_S \nabla f(x^k) \|_M^2\\
& = f(x^k) - \ve{\nabla f(x^k)}{Z_S \nabla f(x^k)} + \frac12 \nabla f(x^k)^T Z_S M Z_S \nabla f(x^k)\\
& = f(x^k) - \ve{\nabla f(x^k)}{Z_S \nabla f(x^k)} + \frac12 \nabla f(x^k)^T Z_S \nabla f(x^k)\\
& = f(x^k) - \frac12 \ve{\nabla f(x^k)}{Z_S \nabla f(x^k)}.
\tagthis
\label{eq:rcddecrease}
\end{align*}
Then, we obtain the following strict decrease for the objective function in the conditional expectation:
\begin{align}
\label{expdecreaseM}
\E[f(x^{k+1}) | {\cal F}_k] & \leq  f(x^k) - \frac12 \| \nabla f(x^k)\|^2_Z,
\end{align}
provided that $x^k$ is not optimal. This holds since we assume that  $Z \succ 0$ on  $\R^n \setminus \range(A^T)$ and since any feasible $x$ satisfying $\nabla f(x) \in \range(A^T)$ is optimal for the original problem. Therefore, \RSD  algorithm belongs to the class of descent methods.  


\subsection{Computation cost per-iteration for \RSD\!\!}
\label{sec:costrsd}
It is easy to observe that if the cost of updating the gradient $\nabla f$ is negligible, then the cost per iteration in \RSD  is given by the computational effort of finding the solution of the subproblem. The sketch sampling $\sS$ can be completely dense (e.g. Gaussian random matrix) or can be extremely sparse (e.g. a few columns of the identity matrix).

\textit{Case 1: dense sketch matrix $S$}.  In this case, since we assume $p \ll n$ (in fact we usually choose $p$ of order ${\mathcal{O}(m)}$ or even smaller), then the computational cost per-iteration in the update \eqref{eq:rcdupdate} is linear in $n$ (more precisely of order ${\mathcal{O}(p m n)}$) plus the cost of computing the matrix $S^T M S \in \R^{p \times p}$. Clearly, if  $M$ is also a dense matrix, then the cost of computing the matrix $S^T M S$ is quadratic in $n$.  However, it can be reduced substantially,  that is the cost of computing this matrix depends linearly on $n$,  when e.g.  we have available a decomposition of the matrix $M$ as $M = \bar{M}^T \bar{M}$, with $\bar{M} \in \R^{\bar{p} \times n}$ and $\bar{p} \ll n$, or $M$ is sparse.

\textit{Case 2: sparse sketch matrix $S$}.  For simplicity, we can  assume that $S$ is chosen as few columns of the identity matrix and thus obtaining a coordinate descent type  method. In this case, the cost per-iteration of \RSD is independent of the problem size $n$. For example, the cost of computing $(AS)^\dagger$ is ${\mathcal{O}(m^2p)}$, while the cost of computing $(P_S^T S^T M S P_S)^\dagger$ is ${\mathcal{O}(p^3)}$.

In conclusion, in all situations the  iteration \eqref{eq:rcdupdate} of \RSD is much computationally cheaper (at least one order of magnitude) than  the iteration \eqref{eq:gupdate} corresponding to the full gradient. Based on the decrease of the objective function \eqref{expdecreaseM} we can derive different convergence rates for our algorithm \RSD  depending on the assumptions imposed on the objective function $f$.


\subsection{Convergence rate: smooth case}
\label{sec_rcd_lip}
We derive in this section the convergence rate of the sequence generated by the \RSD algorithm when the objective function is only smooth (Assumption \ref{ass_lip}). Recall that in the non-convex settings a feasible $x^*$ is a stationary point for optimization problem \eqref{problem} if $\nabla f(x^*) \in \range(A^T)$. On the other hand, for any feasible  $x$ we have the unique decomposition of $\nabla f(x) \in \R^n$:
\[  \nabla f(x)  = A^T \lambda + \nabla f(x)_{\perp}, \quad \text{where} \quad \lambda \in \R^m, \; \nabla f(x)_{\perp} \in \ker(A).  \]  It is clear that if a feasible $x$ satisfies  $\nabla f(x)_{\perp} = 0$, then such an $x$ is a stationary point for \eqref{problem}. In conclusion, a good measure of optimality  for a feasible $x$ is described in terms of $\|\nabla f(x)_{\perp}\|$. The theorem below provides a convergence rate for the sequence generated by \RSD in terms of this optimality measure:
\begin{theorem}
\label{th0}
Let $f$ be bounded from below, i.e. there exists $\bar{f} > - \infty $ such that we have  $\min_{x \in x^0 + \ker(A)} f(x) \geq \bar{f}$ and  Assumptions \ref{ass_lip} and \ref{ass_Z} hold. Then, the iterates of \RSD have the following sublinear convergence rate in expectation:
\begin{equation}
\label{estimate0}
\min_{0 \leq l \leq k-1}  \E[\| \nabla f(x^l)_{\perp}\|^2_Z]  \leq \frac{2 (f(x^0) - \bar{f})}{k}.
\end{equation}
\end{theorem}

\begin{proof}
Taking expectation over the entire history ${\cal F}_k$ in \eqref{expdecreaseM} we get:
\begin{align}
\label{eq:descent_smooth}
\phi_{k+1}  \leq  \phi_k - \frac12 \E[\| \nabla f(x^k)\|^2_Z].
\end{align}
Summing the previous relation and using that $f$ is bounded from below we further get:
\begin{align*}
\sum_{l=0}^{k-1} \E[\| \nabla f(x^l)\|^2_Z] \leq 2(\phi_0 - \phi_k) \leq 2(\phi_0 - \bar{f}).
\end{align*}
Using the unique decomposition $\nabla f(x^l)  = A^T \lambda^l + \nabla f(x^l)_{\perp}$ for all $l$ and since $\ker(Z) = \range(A^T)$, then we obtain $\| \nabla f(x^l)\|^2_Z = \| \nabla f(x^l)_{\perp}\|^2_Z$. Therefore, taking the limit as $k \to \infty$ we obtain the asymptotic convergence $\lim_{k \to \infty} \E[\| \nabla f(x^k)_{\perp}\|^2_Z] =0$. 
Moreover, since $Z \succ 0 $ on $\ker(A)$ and $\nabla f(x^l)_{\perp} \in \ker(A)$ we also get:
\[  \min_{0 \leq l \leq k-1}  \E[\| \nabla f(x^l)_{\perp}\|^2_Z]  \leq \frac{2 (f(x^0) - \bar{f})}{k}, \]
which concludes our statement.
\end{proof}


\subsection{Convergence rate: smooth convex case}
\label{sec_rcd_clip}
\noindent In order to estimate the rate of convergence of our
algorithm when the objective function $f$ is smooth and convex, we introduce
the following distance that takes into account that our algorithm
is a descent method:
\begin{equation}
{\cal R}(x^0) = \max_{\{x \in x^0+\ker(A): f(x) \leq f(x^0)\}} \; \min_{x^* \in X^*} \| x - x^*\|_{Z}^*,
\label{eq:R2}
\end{equation}
which measures the size of the sublevel set of $f$ given by $x^0$. We
assume that this distance is finite for the initial iterate $x^0$. In the next theorem we prove sublinear convergence in expected value of the objective function  for the smooth convex case:
\begin{theorem}
\label{th1}
Let the objective function $f$ be convex and  Assumptions \ref{ass_lip} and \ref{ass_Z} hold. Then, the iterates generated by \RSD have the following sublinear convergence rate in the expected value of the objective function:
\begin{equation}
\label{estimate}
\phi_k - f^* \le \frac{2 {\cal R}^2(x^0)}{k + 2 {\cal R}^2(x^0)/(f(x^0) - f^*)}.
\end{equation}
\end{theorem}

\begin{proof}
Recall that all our iterates are feasible, i.e. $x^k \in x^0 + \ker(A)$.  From
convexity of $f$ and the definition of the norm $\|\cdot\|_{Z}$ on the subspace $\ker(A)$,  we get:
\begin{align*}
f(x^l) - f^* & \leq   \langle \nabla f(x^l), x^l - x^* \rangle  \overset{\eqref{eq:CS}}{\leq}  \|x^l - x^*\|_{Z}^*   \|\nabla f(x^l)\|_{Z} \quad  \forall  x^* \in X^*,  \; l \geq 0.
\end{align*}
Since the previous chain of inequalities hold for any optimal point $x^* \in X^*$, we get further:
\begin{align*}
f(x^l) - f^* & \leq \min_{x^* \in X^*}  \|x^l - x^*\|_{Z}^*   \|\nabla f(x^l)\|_{Z} \overset{\eqref{eq:rcddecrease}+\eqref{eq:R2}}{\leq} {\cal R}(x^0) \cdot  \|\nabla f(x^l)\|_{Z} \quad \forall l \geq 0.
\end{align*}
Combining this inequality with \eqref{expdecreaseM}, we obtain:
\[ f(x^l) - \E \left[ f(x^{l+1}) \;|\; {\cal F}_l \right]  \geq  \frac{(f(x^l) - f^*)^2}{2
{\cal R}^2(x^0)}, \]
or equivalently
\begin{align*}
\E \left[ f(x^{l+1}) \;|\; {\cal F}_l \right]  - f^* \le f(x^l)- f^* - \frac{(f(x^l) - f^*)^2}{2 {\cal R}^2(x^0)}.
\end{align*}
Taking the expectation of both sides of this inequality  in
${\cal F}_{l}$ and denoting $\Delta_l = \phi_l - f^*$ leads to:
\begin{equation*}
\Delta_{l+1} \leq \Delta_l - \frac{\Delta_l^2}{2 {\cal R}^2(x^0)}.
\end{equation*}
Dividing both sides of this inequality with $\Delta_{l}
\Delta_{l+1}$ and taking into account that $\Delta_{l+1} \le
\Delta_l$ (see  \eqref{expdecreaseM}), we obtain:
\begin{equation*}
\frac{1}{\Delta_{l}} \le \frac{1}{\Delta_{l+1}} - \frac{1}{2 {\cal R}^2(x^0)} \quad \forall l \geq 0.
\end{equation*}
Adding these inequalities from $l=0, \cdots, k-1$ we get  the
following inequalities $0 \leq \frac{1}{\Delta_{0}} \le
\frac{1}{\Delta_{k}} - \frac{k}{2{\cal R}^2(x^0)}$, from which we
obtain the desired statement.
\end{proof}

\subsection{Convergence rate: smooth strongly convex case}
\label{sec_rcd_sc}
\noindent In addition to the smoothness assumption, we
now assume that the function $f$ is strongly convex with
respect to the extended norm $\|\cdot\|_{Z}^*$ with strong convexity
parameter $\sigma_{Z} >0$ on the subspace $x^0+\ker(A)$:

\begin{assumption}
\label{ass_scZ}
We assume that the objective function $f$ is strongly  convex on the subspace $x^0+\ker(A)$, that is there exists a parameter $\sigma_{Z} >0$ satisfying the following inequality:
\begin{equation}
\label{strongq}
f(x) \ge f(y) + \langle \nabla f(y), x -y \rangle +
\frac{\sigma_{Z}}{2} \left( \| x-y \|_{Z}^* \right)^2 \quad \forall x, y \in x^0+\ker(A).
\end{equation}
\end{assumption}
Note that if $f$ is strongly convex function everywhere in $\R^n$, that is  there exists a positive definite  matrix $G$ such that:
\begin{equation}
\label{eq:G}
f(x) \geq f(y)+\ve{\nabla f(y)}{x-y} +\frac12 (y-x)^T G (y-x) \quad \forall x,y \in \R^n,
\end{equation}
then using the definition of the dual norm $(\|x\|_Z^*)^2 = x^T Z^\dagger x$ (see Lemma \ref{lem4}) we  have that \eqref{strongq} also holds for some $\sigma_Z$ satisfying:
\[ G \succeq \sigma_Z Z^\dagger \quad \text{on} \quad \ker(A) \;\; (\text{or equivalently on} \; \R^n \setminus \range(A^T)).  \]
Since $x^T Z x =0$ for all $x \in \range(A^T)$ (see Lemma \ref{lema1}), then also $x^T Z^\dagger x =0$ for all $x \in \range(A^T)$. In conclusion, the matrix inequality $G \succeq \sigma_Z Z^\dagger$ holds automatically on $\range(A^T)$ for any constant $\sigma_Z$, and consequently we can  define $\sigma_Z$ as the largest positive constant satisfying everywhere on $\R^n$ the matrix inequality:
\[ G \succeq \sigma_Z Z^\dagger.  \]
This shows that  Assumption \ref{ass_scZ} is less restrictive than requiring strong convexity for $f$ everywhere in $\R^n$ as in \eqref{eq:G}. Next, we prove that the strong convexity parameter  $\sigma_Z$ is bounded from above:

\begin{lemma}
Under Assumptions \ref{ass_lip}, \ref{ass_Z} and \ref{ass_scZ}  the strong convexity parameter $\sigma_Z$ defined in \eqref{strongq} is bounded above by:
\begin{equation}
\sigma_Z \leq  \lambda_{\max} (M^{1/2}ZM^{1/2}) \leq 1.
\label{asdfasdfasdfa}
\end{equation}	
\end{lemma}

\begin{proof}
By the Lipschitz continuous gradient inequality (see Assumption \ref{ass_lip}) and the strong convexity inequality  (see Assumptions \ref{ass_scZ}) we have that $\sigma_Z Z^\dagger \preceq M$ on $\ker(A)$ (or equivalently on $\R^n \setminus \range(A^T)$). Since $x^T Z x =0$ for all $x \in \range(A^T)$ (see Lemma \ref{lema1}), then also $x^T Z^\dagger x =0$ for all $x \in \range(A^T)$. In conclusion, the matrix inequality $\sigma_Z Z^\dagger \preceq M$ holds automatically on $\range(A^T)$. Therefore, we get the following matrix inequality valid on $\R^n$:
\[  \sigma_Z Z^\dagger \preceq M.  \]
Pre- and post-multiplying the previous matrix inequality with $Z M^{1/2}$ leads to:
\begin{align*}
\sigma_Z  M^{1/2}  Z Z^\dagger Z M^{1/2} \preceq M^{1/2} Z M Z M^{1/2},
\end{align*}
or equivalently
\begin{align*}
\sigma_Z  M^{1/2}  (Z Z^\dagger Z) M^{1/2}   \preceq  M^{1/2} Z  (M^{1/2} M^{1/2}) Z M^{1/2}.
\end{align*}
Using the basic properties of the pseudo-inverse we obtain:
\[ \sigma_Z \; M^{1/2} Z M^{1/2} \preceq (M^{1/2}ZM^{1/2}) (M^{1/2} Z M^{1/2}). \]
Therefore, if we denote by $\Upsilon = M^{1/2}ZM^{1/2} \succeq 0$, then we get that $\Upsilon^2 - \sigma_Z \Upsilon \succeq 0$ and thus for any eigenvalue $\lambda$ of $\Upsilon$ it holds that $\lambda^2 -\sigma_Z \lambda \geq 0$ or equivalently $\sigma_Z \leq \lambda$. It remains to show that $\lambda_{\max} (\Upsilon) \leq 1$.  For this, we recall that according to Lemma \ref{lema1} we know that $Z_S = Z_S M Z_S $. By utilizing the fact that $M$ is symmetric and positive-definite, we can notice that
\begin{align*}
M^{1/2} Z_S M^{1/2} &= M^{1/2} Z_S M Z_S M^{1/2}  = (M^{1/2} Z_S M^{1/2}) (M^{1/2} Z_S M^{1/2}).
\end{align*}
Therefore, all the eigenvalues of $M^{1/2} Z_S M^{1/2}$ belongs to the set $\{0, 1\}$. Further, by the  definition of $Z$ in \eqref{eq:def_of_Z} and using the convexity of the function $\lambda_{\max}$ on the set of positive semidefinite matrices,  we have:
\begin{align*}
\lambda_{\max} (M^{1/2} Z M^{1/2}) & = \lambda_{\max} (\E_S [M^{1/2} Z_S M^{1/2}]) \leq \E_S[ \lambda_{\max} (M^{1/2} Z_S M^{1/2})]   \leq 1,
\end{align*}	
which completes our proof.
\end{proof}	

\noindent  We now derive a linear convergence estimate for our
algorithm \RSD under this additional strong convexity assumption on the subspace $x^0+\ker(A)$:
\begin{theorem}
\label{th3sc}
Under Assumptions \ref{ass_lip}, \ref{ass_Z} and \ref{ass_scZ} the sequence generated by \RSD satisfies the following linear convergence rate for the expected value of the objective function:
\begin{equation}
\label{strongconv}
\phi_k - f^* \le (1 - \sigma_{Z})^k \left(f(x^0) - f^* \right).
\end{equation}
\end{theorem}

\begin{proof}
Given $x^k$, taking the conditional expectation in \eqref{expdecreaseM} over the random matrix $S$  leads to the following inequality:
\begin{equation}
\label{eq:ineqrcdsc}
2\left( f(x^k) -  E \left[ f(x^{k+1}) \;|\; x^k \right] \right) \ge
\|\nabla f(x^k)\|_{Z}^{2}.
\end{equation}
On the other hand,  consider the minimization of the right
hand side in \eqref{strongq} over $x \in x^0+\ker(A)$, and denote $x(y)$ its optimal solution. Using the definition of the dual norm  $\| \cdot \|_{Z}^*$ in the subspace $\ker(A)$, one can see that $x(y)$ satisfies the following optimality conditions:
\[ \exists \mu  \;\; \text{s.t.}: \;\;\; \nabla f(y) + \sigma_{Z} Z^\dagger (x(y) - y) + A^T\mu=0 \quad \text{and} \quad x(y) \in x^0+\ker(A). \]
Combining these optimality conditions with  the well-known property of the pseudo-inverse, that is $Z^\dagger Z Z^\dagger = Z^\dagger$, we get that the optimal value of this minimization problem has the  following  expression:
\[ f(y) - \frac{1}{2 \sigma_{Z}} \|\nabla f(y)\|_{Z}^{2}. \]
Therefore, minimizing both sides of inequality
\eqref{strongq} over $x \in x^0+\ker(A)$, we have:
\begin{equation}
\label{strongq_less}
\|\nabla f(y)\|_{Z}^{2} \ge 2 \sigma_{Z} (f(y) - f^*)
\quad \forall y \in x^0+ \ker(A)
\end{equation}
and for $y = x^k$ we get:
\begin{equation*}
\|\nabla f(x^k)\|_{Z}^{2} \ge 2 \sigma_{Z} \left( f(x^k) -
f^* \right).
\end{equation*}
Combining the  inequality \eqref{eq:ineqrcdsc} with the previous one, and taking
expectation in ${\cal F}_{k-1}$ on both sides, we arrive at the statement of
the theorem.
\end{proof}

\begin{remark}
From the proof of Theorem  \ref{th3sc} it follows that we can further relax the strong convexity assumption, that is instead of  \eqref{strongq} it is sufficient to require \eqref{strongq_less} to hold on $x^0+ \ker(A)$. The reader should note that an inequality of the form  \eqref{strongq_less} is known in the optimization literature  as the Polyak-Lojasiewicz (PL) condition (see e.g. \cite{KarNut:16} for a recent exposition), and the proof above shows that algorithm \RSD converges linearly for smooth convex functions satisfying only the PL condition. Since functions satisfying the PL inequality need not be convex, linear convergence of \RSD method to the global optimum extends beyond the realm of convex functions. More precisely, is is easy to see that  the convergence result of Theorem \ref{th0} can be strengthen, that is we can prove linear convergence  in the expected values of the objective function  for the iterates of algorithm \RSD provided that additionally the PL type condition \eqref{strongq_less} holds (we just need to combine the inequalities \eqref{eq:descent_smooth}  and \eqref{strongq_less}).
\end{remark}

\noindent Note that in special cases, where complexity bounds are known for \RSD, such as optimization problems with a single linear coupled constraint, our theory recovers the best known bounds (see e.g. the convergence analysis in \cite{FroRei:15,NecNes:11}). For example, in the smooth convex case choosing for the sketch matrix $S$ at least $p \geq 2$ columns of the identity matrix, then combining \eqref{eq:Zij} with Theorem \ref{th1} we recover the convergence rate of coordinate descent algorithm from \cite[Theorem 4.1]{NecNes:11} for the problem with a separable objective function and a single linear constraint. Similarly, for the strongly convex case our convergence analysis recovers \cite[Theorem 4.2]{NecNes:11}. In conclusion, to our knowledge, this is the first complete convergence analysis  of a general random sketch descent  algorithm, for which coordinate descent method is a particular case,  for solving optimization problems with multiple linear coupled  constraints.


\section{Accelerated random sketch descent algorithm}
For the accelerated variant of Algorithm \RSD let us first  define the following constant:
\begin{equation}
\label{eq:arcd_nu}
\nu_{\max}  = \max_{u \in \ker(A), u \not=0} \frac{\Exp [ (\| Z_S u\|^*_{Z})^2 ]}{\|u\|_Z^2} = \max_{u \in \ker(A), u \not=0} \frac{\Exp [\| Z_S u\|^2_{Z^\dagger}]}{\|u\|_Z^2}.
\end{equation}
Let us now consider any constant parameter $\nu \geq \nu_{\max}$. The Accelerated Random Sketch Descent (A-\RSD\!) scheme is depicted  in Algorithm~\ref{alg:arcd}:
\begin{algorithm}
\caption{Algorithm A-\RSD}
\label{alg:arcd}
\begin{algorithmic}[1]
\STATE {\bf Input:} Positive sequences $\{\alpha_k\}_{k=0}^\infty,  \{\beta_k\}_{k=0}^\infty,  \{\gamma_k\}_{k=0}^\infty$
\STATE Choose  $x^0 \in \R^n$ such that $ A x^0 = b$ and set $v^0 = x^0$
\FOR {  $k \geq 0$}
\STATE
sample $S \sim {\sS}$  and perform the following updates:
\STATE  $y^k = \alpha_k v^k + (1-\alpha_k) x^k $
\STATE $x^{k+1} = y^k -  Z_S \nabla f(y^k)  $
\STATE $v^{k+1} = \beta_k v^k + (1 - \beta_k) y^k - \gamma_k Z_S \nabla f(y^k)$
\ENDFOR
\end{algorithmic}
\end{algorithm}


\subsection{Computation cost per-iteration for A-\RSD\!\!} 
It is easy to observe that the computational  cost for updating the sequence $x^k$ is comparable to the one corresponding to \RSD algorithm. Therefore, the conclusions regarding the cost per-iteration from Section \ref{sec:costrsd} corresponding to \RSD are also valid here. Note that the accelerated variant also requires updating two additional sequences $y^k$ and $v^k$, which requires computations with full vectors in $\R^n$. However, for structured optimization problems we can avoid the addition of full vectors in $\R^n$ and still keep the cost per-iteration of A-\RSD comparable to that of \RSD\!.  More precisely, we can efficiently implement the updates of  A-\RSD algorithm without  full-dimensional vector operations when the sketch matrix  $S$ is sparse and when we can efficiently compute:
\[  \nabla f(\alpha v + \beta u)  \qquad \forall \alpha, \beta \in \R \;\; \text{and} \;\;  v,u \in \R^n. \]
Note that gradient evaluation in such points  is computationally easy  when  $f$ has a special structure, e.g.  of the form $f(x) = g(Ex)$, where $E$ is a sparse matrix \cite{FerRic:15}. Objective functions of this form includes many generalized linear models, such as  logistic regression, least squares, etc. In Appendix A  we provide efficient implementations  of the updates of A-\RSD for these settings.


\subsection{Basic properties of A-\RSD\!}
Before deriving convergence rates for A-\RSD we analyze some basic properties of this algorithm. First, we  prove  that the newly introduced constant  $\nu_{\max}$ is bounded, thus finite:
\begin{lemma}
Under Assumptions \ref{ass_lip}, \ref{ass_Z} and \ref{ass_scZ} we have:
\[ 0 < \sigma_Z \leq  \nu_{\max} \leq \lambda_{\max} (M^{-1/2} Z^\dagger M^{-1/2}) < \infty. \]
\end{lemma}

\begin{proof}
If we denote  $c = \lambda_{\max} (M^{-1/2} Z^\dagger M^{-1/2})$, then it follows that  $Z^{\dagger} \preceq c M$.  Using this matrix inequality in the definition of $\nu_{\max}$ we have:
\begin{align*}
\frac{\Exp [\| Z_S u\|^2_{Z^\dagger}]}{\|u\|_Z^2} & = \frac{\Exp [u^T Z_S Z^\dagger Z_S u]}{u^T Z u} \leq  \frac{\Exp [c \cdot  u^T Z_S M Z_S u]}{u^T Z u} \\
& = c \frac{\Exp [u^T Z_S u]}{u^T Z u} = c \quad \forall u \in \ker(A), u \not=0.
\end{align*}
This proves that $\nu_{\max} \leq c < \infty$ provided that Assumptions \ref{ass_lip} and \ref{ass_Z} hold. Now, we will show that
$\sigma_Z \leq \nu_{\max}$ if additionally Assumption \ref{ass_scZ} holds.
Indeed, from Jensen inequality we have:
\begin{align*}
\nu_{\max} &= \max_{u \in \ker(A), u \not=0} \frac{\Exp [\| Z_S u\|^2_{Z^\dagger}]}{\|u\|_Z^2} \geq \max_{u \in \ker(A), u \not=0} \frac{\| \Exp [Z_S] u\|^2_{Z^\dagger}}{\|u\|_Z^2}\\
& =  \max_{u \in \ker(A), u \not=0} \frac{\|  Z   u\|^2_{Z^\dagger}}{\|u\|_Z^2}
= \max_{u \in \ker(A), u \not=0} \frac{\|u\|^2_{Z }}{\|u\|_Z^2}  = 1 \overset{\eqref{asdfasdfasdfa}}{\geq } \sigma_Z,
\end{align*}
which concludes the proof.
\end{proof}

EXAMPLE 6 cont. \textit{For the optimization problem considered in Example \ref{example} we can easily compute a good upper approximation for} $\nu_{\max}$:
\begin{align*}
\nu_{\max} &= \!\! \max_{u \in \ker(A), u \not=0} \!\! \frac{\Exp[u^T Z_{(i,j)} Z^\dagger Z_{(i,j)} u]}{u^T Z u} = \!\! \max_{u \in \ker(A), u \not=0} \!\! \frac{\Exp \left[ \frac{2(n-1)L}{n(L_i + L_j)}  u^T  Z_{(i,j)} u \right]}{u^T Z u}  \leq \max_{i<j} \frac{2(n-1)L}{n(L_i + L_j)},
\end{align*}
\textit{where we used that $(e_i - e_j)^T (I_n - 1/n \ e e^T) (e_i - e_j) =2$. This relation  shows that $\nu_{\max} \sim  L \; (:= \sum_i L_i)$  and consequently it is related to a global Lipschiz type constant  for the gradient of  $f$.}

For simplicity of the exposition let us also denote:
\[ g^k = - Z_S \nabla f(y^k) \quad \left(=-S P_S  (P^T_S S^T M S P_S)^\dagger P^T_S S^T \nabla f(y^k)\right). \]
From the updates of A-\RSD we can also show a descent property  for the conditional expectation $\Exp [f(x^{k+1})| {\cal F}_k]$. Indeed, from our updates and Assumption \ref{ass_lip} we have:
\begin{align*}
f(x^{k+1}) & = f(y^{k} + g_k) \leq f(y^k) + \ve{\nabla f(y^k)}{g_k} + \frac12 \| g_k\|_M^2.
\end{align*}
Taking now the conditional expectation with respect to random choice $S$ and using that $Z_S M Z_S = Z_S$ (see Lemma \ref{lema1}) we obtain:
\begin{align*}
\Exp [f(x^{k+1})| {\cal F}_k] & \leq f(y_k)
+\ve{\nabla f(y_k)}  {\Exp[g_k | {\cal F}_k] } +\frac12 \Exp [ \| g_k\|^2_M | {\cal F}_k] \\
& = f(y_k) +\ve{\nabla f(y_k)}  {\Exp[-Z_S \nabla f(y^k) | {\cal F}_k] } +\frac12 \Exp [ \| Z_S \nabla f(y^k)\|^2_M | {\cal F}_k] \\
&= f(y_k) - \|\nabla f(y_k)\|_Z^2 + \frac12  \| \nabla f(y^k) \|^2_Z   = f(y_k ) - \frac12 \|\nabla f(y_k)\|_Z^2.
\tagthis
\label{eq:arcd_desc}
\end{align*}
Moreover, the sequences $x^k, y^k$ and $v^k$ satisfies $x^k- x^* \in \ker(A)$, $y^k - x^* \in \ker(A)$ and $v^k - x^* \in \ker(A)$, and consequently also $x^k - y^k \in \ker(A)$. Moreover, since  $\range(A^T) = \ker(Z)$ (see Lemma \ref{lem3}), then $Z^\dagger Z$ is a projection matrix onto $\ker(A)$, that is  $Z^\dagger Z u = u$ for all $u \in  \ker(A)$, and thus the following holds:
\begin{align}
\label{eq:arcd_zz}
Z^\dagger Z (x^* - y^k) = x^* - y^k \quad \text{and} \quad  Z^\dagger Z (x^k - y^k) = x^k - y^k.
\end{align}
For any optimal point $x^*$ let us also define the sequence:
\begin{align}
\label{rk2}
r_k^2 = \|v^k - x^* \|_{Z^\dagger}^2.
\end{align}
Based on the previous discussion, we can show the following descent property for the sequence $r_k$ of Algorithm A-\RSD that holds also for the case $\sigma_Z=0$:

\begin{lemma}
\label{lem_arcd}
Under Assumptions \ref{ass_lip}, \ref{ass_Z} and \ref{ass_scZ} and any choices for the sequences  $\{\alpha_k\}_{k=0}^\infty \in (0, \ 1]$, $\{\beta_k\}_{k=0}^\infty \in (0, \ 1]$ and $\{\gamma_k\}_{k=0}^\infty \in (0, \ \infty)$ the Algorithm A-\RSD produces a sequence of points $(x_k, y_k, v_k)$ such that the following descent inequality holds:
\begin{align*}
& \Exp [r_{k+1}^2 + 2 \gamma_k^2 \nu  (f(x^{k+1}) - f^*) | {\cal F}_k ]  \leq   \beta_k \left( r_k^2 +  2 \gamma_k \frac{1-\alpha_k}{\alpha_k} (f(x^k) - f^*) \right) \tagthis
\label{eq:recursion_arcd}\\
& \qquad  + (1 - \beta_k - \gamma_k \sigma_Z) \|y^k - x^* \|_{Z^\dagger}^2 + \left( 2 \gamma_k^2 \nu  - 2 \gamma_k - 2 \gamma_k \beta_k \frac{1-\alpha_k}{\alpha_k} \right) (f(y^k) - f^*).
\end{align*}
\end{lemma}

\begin{proof}
Using the definition of $r_{k+1}$  we have:
\begin{align*}
r_{k+1}^2 & \overset{\eqref{rk2}}{=}
\|v^{k+1} - x^* \|_{Z^\dagger}^2
 = \| \beta_k v^k + (1-\beta_k) y^k - x^* + \gamma_k g_k \|_{Z^\dagger}^2\\
&= \| \beta_k v^k + (1-\beta_k) y^k - x^* \|_{Z^\dagger}^2 + \gamma_k^2 \|g_k\|_{Z^\dagger}^2 \\
& \qquad + 2 \gamma_k \left( \beta_k v^k + (1-\beta_k) y^k - x^* \right)^T Z^\dagger g_k \\
& \leq  \beta_k \|v^k - x^*\|_{Z^\dagger}^2 + (1-\beta_k) \|y^k - x^*\|_{Z^\dagger}^2 + \gamma_k^2 \|g_k\|_{Z^\dagger}^2 \\
& \qquad + 2\gamma_k \left( \beta_k v^k + (1-\beta_k) y^k - x^* \right)^T Z^\dagger g_k,
\end{align*}
where  in the last inequality we used the convexity of the norm and the fact that $\beta_k \in [0, \ 1]$. Taking now the conditional expectation with respect to ${\cal F}_k$ we get:
\begin{align*}
\Exp[r_{k+1}^2 | {\cal F}_k] & \leq  \beta_k \|v^k - x^*\|_{Z^\dagger}^2 + (1-\beta_k) \|y^k - x^*\|_{Z^\dagger}^2 + \gamma_k^2 \Exp[\|-Z_S \nabla f(y^k)\|_{Z^\dagger}^2 | {\cal F}_k] \\
& \qquad + 2\gamma_k \left( \beta_k v^k + (1-\beta_k) y^k - x^* \right)^T Z^\dagger (- Z \nabla f(y^k)) \\
& \overset{\eqref{eq:arcd_nu}}{\leq} \beta_k r_k^2 + (1-\beta_k) \|y^k - x^*\|_{Z^\dagger}^2 + \gamma_k^2 \nu \|\nabla f(y^k)\|_{Z}^2 \\
& \qquad + 2 \gamma_k \left(x^* - \beta_k v^k - (1-\beta_k) y^k  \right)^T Z^\dagger Z \nabla f(y^k) \\
& = \beta_k r_k^2 + (1-\beta_k) \|y^k - x^*\|_{Z^\dagger}^2 + \gamma_k^2 \nu \|\nabla f(y^k)\|_{Z}^2  \\
& \qquad + 2 \gamma_k \left(x^* - y^k \right) Z^\dagger Z \nabla f(y^k) - 2 \gamma_k  \beta_k \left(v^k -  y^k \right)^T Z^\dagger Z \nabla f(y^k) \\
& = \beta_k r_k^2 + (1-\beta_k) \|y^k - x^*\|_{Z^\dagger}^2 + \gamma_k^2 \nu \|\nabla f(y^k)\|_{Z}^2  \\
& \qquad + 2 \gamma_k \left(x^* - y^k \right) Z^\dagger Z \nabla f(y^k) - 2 \gamma_k  \beta_k \frac{1 - \alpha_k}{\alpha_k} \left(y^k -  x^k \right)^T Z^\dagger Z \nabla f(y^k)\\
& \overset{\eqref{eq:arcd_desc}}{\leq} \beta_k r_k^2 + (1-\beta_k) \|y^k - x^*\|_{Z^\dagger}^2 + 2 \gamma_k^2 \nu  \left( f(y^k)  - \Exp[f(x^{k+1}) | x^k] \right)\\
& \qquad  + 2 \gamma_k \left(x^* - y^k \right) Z^\dagger Z \nabla f(y^k) - 2 \gamma_k  \beta_k \frac{1 - \alpha_k}{\alpha_k} \left(y^k -  x^k \right)^T Z^\dagger Z \nabla f(y^k).
\end{align*}
Rearranging the terms,  we get:
\begin{align*}
& \Exp [r_{k+1}^2 + 2 \gamma_k^2 \nu  (f(x^{k+1}) - f^*) | x^k ]\\
& \quad \leq \beta_k r_k^2 + (1-\beta_k) \|y^k - x^* \|_{Z^\dagger}^2 +  2 \gamma_k^2 \nu (f(y^k ) - f^*) \\
& \qquad   +  2 \gamma_k \left(x^* - y^k\right)^T Z^\dagger Z \nabla f(y^k) + 2 \gamma_k \beta_k \frac{1-\alpha_k}{\alpha_k} \left(x^k - y^k\right)^T Z^\dagger Z \nabla f(y^k)\\
&\quad  \overset{\eqref{eq:arcd_zz}}{\leq} \beta_k r_k^2 + (1-\beta_k) \|y^k - x^* \|_{Z^\dagger}^2 +  2 \gamma_k^2 \nu (f(y^k ) - f^*) \\
& \qquad   +  2 \gamma_k \left( x^* - y^k \right)^T  \nabla f(y^k) + 2 \gamma_k \beta_k \frac{1-\alpha_k}{\alpha_k} \left( x^k - y^k \right)^T \nabla f(y^k)\\
& \quad \overset{\eqref{strongq}}{\leq} \beta_k r_k^2 + (1-\beta_k) \|y^k - x^* \|_{Z^\dagger}^2 +  2 \gamma_k^2 \nu (f(y^k) - f^*) \\
& \qquad   +  2 \gamma_k \left( f^* - f(y^k) - \frac{\sigma_Z}{2} \| y^k - x^* \|^2_{Z^\dagger} \right) + 2 \gamma_k \beta_k \frac{1-\alpha_k}{\alpha_k} \left(x^k - y^k\right)^T \nabla f(y^k).
\end{align*}
Note that the previous derivations also hold without Assumption \ref{ass_scZ}, that is  we use the strong convexity inequality \eqref{strongq} with $\sigma_Z=0$. Using now the convexity of the function $f$ and that $\alpha_k \in (0, \ 1]$, we further get:
\begin{align*}
& \Exp [r_{k+1}^2 + 2 \gamma_k^2 \nu  (f(x^{k+1}) - f^*) | x^k ]\\
& \quad \leq \beta_k r_k^2 + (1-\beta_k - \gamma_k \sigma_Z) \|y^k - x^* \|_{Z^\dagger}^2 \\
& \qquad + (2 \gamma_k^2 \nu  - 2 \gamma_k) (f(y^k) - f^*) +  2 \gamma_k \beta_k \frac{1-\alpha_k}{\alpha_k} (f(x^k) - f(y^k))\\
& \quad =  \beta_k \left( r_k^2 +  2 \gamma_k \frac{1-\alpha_k}{\alpha_k} (f(x^k) - f^*) \right) + (1 - \beta_k - \gamma_k \sigma_Z) \|y^k - x^* \|_{Z^\dagger}^2 \\
& \qquad + \left( 2 \gamma_k^2 \nu  - 2 \gamma_k - 2 \gamma_k \beta_k \frac{1-\alpha_k}{\alpha_k} \right) (f(y^k) - f^*),
\end{align*}
which concludes our statement.
\end{proof}
Based on the previous descent property we can derive different convergence rates for our algorithm A-\RSD depending on the assumptions imposed on the objective function $f$.


\subsection{Convergence rate: smooth convex case}
In this section we prove the sublinear convergence rate for A-\RSD (Algorithm~\ref{alg:arcd}) for some choices of the sequences  $\{\alpha_k\}_{k=0}^\infty$, $\{\beta_k\}_{k=0}^\infty$ and $\{\gamma_k\}_{k=0}^\infty$. In particular, the next lemma shows the behavior of $\{\gamma_k\}_{k=0}^\infty$ defined as follows:

\begin{lemma}
\label{lem:gammak}
Let $\{\gamma_k\}_{k=0}^\infty$ be a sequence defined recursively as
$\gamma_0 = \frac1\nu$ and  $\gamma_{k+1}$ be the largest solution of the second order equation:
\begin{align}
\label{asfafawefawa}
\gamma_{k+1}^2 - \tfrac1\nu \gamma_{k+1}  = \gamma_k^2.
\end{align}
Then, $\gamma_k$ satisfies the following inequality:
\begin{align}
\gamma_{k} \geq \frac{k+2}{2\nu}.
\label{asdfawfwaefwa}
\end{align}
\end{lemma}

\begin{proof}
First, we observe that $\{\gamma_k\}_{k=0}^\infty$ is a non-decreasing sequence. Indeed, the largest root of \eqref{asfafawefawa} is given by:
\begin{align}
\gamma_{k+1} = \frac{\frac1\nu + \sqrt{\frac1{\nu^2} +4 \gamma_k^2 }}{2}
\geq  \frac{\sqrt{ 4 \gamma_k^2}}{2} = \gamma_k.
\label{afwdefawefaw}
\end{align}
Next, we have:
\begin{align*}
\tfrac1\nu \gamma_{k+1}  \overset{\eqref{asfafawefawa}}{=} \gamma_{k+1}^2 - \gamma_k^2 = (\gamma_{k+1}  - \gamma_k) (\gamma_{k+1}  + \gamma_k) \overset{\eqref{afwdefawefaw}}{\leq} 2 \gamma_{k+1} (\gamma_{k+1}  - \gamma_k),
\tagthis
\label{asfawefawefwa}
\end{align*}
which implies that
\begin{align*}
\gamma_k + \tfrac1{2\nu} \leq \gamma_{k+1} \quad \Rightarrow \quad
\gamma_k \geq \gamma_0 + k \frac1{2\nu}
= \frac{2+k}{2\nu}.
\end{align*}
This concludes our proof.
\end{proof}

\noindent From \eqref{asdfawfwaefwa} it follows that $\gamma_k \nu \geq 1$ for all $k \geq 0$. Now, we are ready to prove the sublinear convergence of A-\RSD:
\begin{theorem}
Under Assumptions \ref{ass_lip} and \ref{ass_Z} the sequences generated by Algorithm A-\RSD with $\alpha_k = \frac1{\gamma_k \nu} \in (0, \ 1]$, $\beta_k = 1$, $\gamma_0 = \frac1\nu$ and $\gamma_k$ be the largest solution defined by recursion  \eqref{asfafawefawa}, satisfy the following sublinear convergence rate in expectation:
\begin{align*}
\Exp \left[f(x^{k }) - f^* \right] \leq \frac{2 \nu }{(k+1)^2} \min_{x^* \in X^*} \|x^0 - x^* \|_{Z^\dagger}^2 \quad \forall k \geq 1.
\end{align*}
\end{theorem}

\begin{proof}
In the smooth convex case we can use  Lemma \ref{lem_arcd} (i.e. descent relation \eqref{eq:recursion_arcd}) by setting $\sigma_Z = 0$, i.e. we have:
\begin{align*}
\Exp [r_{k+1}^2 + 2 \gamma_k^2 \nu  (f(x^{k+1}) - f^*) | {\cal F}_k ]
&\overset{\eqref{eq:recursion_arcd}}{\leq} \left( r_k^2 +  2 \gamma_k \frac{1-\alpha_k}{\alpha_k} (f(x^k) - f^*) \right)   \\
& \qquad + \left( 2 \gamma_k^2 \nu  - 2 \gamma_k   \frac{1 }{\alpha_k}
\right) (f(y^k) - f^*).
\tagthis \label{eq:afwfweafwaeva}
\end{align*}
Note that  $\alpha_k = \frac1{\gamma_k \nu}$ and hence the last term in \eqref{eq:afwfweafwaeva} vanishes. Thus, we further obtain:
\begin{align*}
\tagthis\label{afsafa}
\Exp [r_{k+1}^2 + 2 \gamma_k^2 \nu  (f(x^{k+1}) - f^*) | {\cal F}_k ]
&\overset{\eqref{eq:afwfweafwaeva} }{\leq}   r_k^2 +  2 \gamma_k \frac{1-\alpha_k}{\alpha_k} (f(x^k) - f^*).
\end{align*}
Moreover, since $\alpha_k = \frac1{\gamma_k \nu}$,  then:
\begin{align}
2 \gamma_k \frac{1-\alpha_k}{\alpha_k} = 2 \gamma_k^2 \nu \left( 1- \frac1{\gamma_k \nu} \right) = 2 \gamma_k^2 \nu  - 2 \gamma_k.
\label{Asdfsadfasdfa}
\end{align}
Plugging \eqref{Asdfsadfasdfa} into \eqref{afsafa} and dividing both sides by $2\nu$ we obtain:
\begin{align*}
\Exp [\tfrac{1}{2\nu}r_{k+1}^2 +    \gamma_k^2    (f(x^{k+1}) - f^*) | {\cal F}_k ] \leq \left(\tfrac{1}{2\nu} r_k^2 + (\gamma_k^2 - \tfrac1{\nu} \gamma_k) (f(x^k) - f^*) \right).
\tagthis
\label{Afweaefweaefa}
\end{align*}
Now, it reminds  to note that $\gamma_{k+1}$ satisfy \eqref{asfafawefawa} and consequently:
\begin{align*}
\Exp [\tfrac{1}{2\nu}r_{k+1}^2 +  (\gamma_{k+1}^2- \tfrac1{ \nu}\gamma_{k+1})   (f(x^{k+1}) - f^*) | {\cal F}_k] \leq  \left(\tfrac{1}{2\nu} r_k^2 + (\gamma_k^2- \tfrac1{ \nu} \gamma_k) (f(x^k) - f^*) \right).
\end{align*}
Taking now the expectation over the entire history in the previous recursion and  unrolling it,  we get:
\begin{align*}
\Exp [  (\gamma_{k}^2 - \tfrac1{ \nu}\gamma_{k}) (f(x^{k}) - f^*)]
& \leq \Exp [\tfrac{1}{2\nu}r_{k}^2 + (\gamma_{k}^2 - \tfrac1{ \nu}\gamma_{k}) (f(x^{k}) - f^*)] \\
& \leq \left(\tfrac{1}{2\nu} r_0^2  + (\gamma_0^2- \tfrac1{\nu} \gamma_0) (f(x^0) - f^*) \right).
\end{align*}
Since we have the second order equation $\gamma_{k}^2 - \tfrac1{\nu} \gamma_{k}  =  \gamma_{k-1}^2 \overset{\eqref{asdfawfwaefwa}}{\geq} \left(\frac{k+1}{2\nu}\right)^2$ for all $k \geq 1$, we get our statement.
\end{proof}


\subsection{Convergence rate: smooth strongly convex case}
We are now ready to state the linear convergence rate for A-\RSD (Algorithm~\ref{alg:arcd}).
\begin{theorem}
\label{Thm:AcceRate}
Under Assumptions \ref{ass_lip}, \ref{ass_Z} and \ref{ass_scZ} the sequences generated by Algorithm A-\RSD with
$\alpha_k = \frac{\gamma_k \sigma_Z}{1 + \gamma_k \sigma_Z} \in (0, \ 1]$,
$\beta_k = 1 - \gamma_k \sigma_Z \in [0, \ 1]$ and
$\gamma_k = \frac{1}{\sqrt{\sigma_Z \nu}} \leq \frac{1}{\sigma_Z}$
satisfy the following linear convergence rate in expectation:
\begin{align*}
\Exp \left[r_{k}^2+  \frac{2}{\sigma_Z} (f(x^{k}) - f^*) \right] & \leq \left( 1 - \sqrt{\frac{\sigma_Z}{\nu}} \right)^k \left( r_0^2 + \frac{2}{\sigma_Z} (f(x^0) - f^*) \right).
\end{align*}
\end{theorem}

\begin{proof}
Note that the choices of $\alpha_k, \beta_k$ and $\gamma_k$ from the theorem guarantee that:
\begin{align*}
2 \gamma_k^2 \nu = 2 \gamma_k \frac{1-\alpha_k}{\alpha_k}, \quad 1 - \beta_k - \gamma_k \sigma_Z \leq  0,  \quad  2 \gamma_k^2 \nu  - 2 \gamma_k - 2 \gamma_k \beta_k \frac{1-\alpha_k}{\alpha_k} =0.
\end{align*}
Using these relations in Lemma \ref{lem_arcd} (i.e. descent relation \eqref{eq:recursion_arcd}), we get:
\begin{align*}
 \Exp [r_{k+1}^2 + 2 \gamma_k^2 \nu  (f(x^{k+1}) - f^*) | {\cal F}_k ]
 &\overset{\eqref{eq:recursion_arcd}}{\leq}   \beta_k \left( r_k^2 +  2 \gamma_k \frac{1-\alpha_k}{\alpha_k} (f(x^k) - f^*) \right).
\end{align*}
After plugging $\alpha_k = \frac{\gamma_k \sigma_Z}{1 + \gamma_k \sigma_Z}  $ and
$\gamma_k = \frac{1}{\sqrt{\sigma_Z \nu}}  $
we further obtain:
\[  \Exp [r_{k+1}^2 + 2/\sigma_Z  (f(x^{k+1}) - f^*) | {\cal F}_k ] \leq  \beta_k \left( r_k^2 +  2/\sigma_Z (f(x^k) - f^*) \right), \]
and taking now the expectation over the entire history we get:
\[ \Exp [r_{k}^2 + 2/\sigma_Z  (f(x^{k}) - f^*)] \leq  \left( \prod_{j=0}^{k-1} \beta_j \right) \left( r_0^2 +  2/\sigma_Z (f(x^0) - f^*) \right), \]
which leads to the statement of our theorem.
\end{proof}

\begin{table}
\caption{Comparison of convergence rates of \RSD and A-\RSD Algorithms.}
\label{tbl:comparison}
\centering
\begin{tabular}{ | m{3.5cm} | m{2cm} | m{2cm} | }
\hline
   &  \RSD & A-\RSD \\
\hline
smooth convex & $\frac{\|x^0 - x^* \|^2_{Z^\dagger}}{k}$ &  $\frac{\nu \|x^0 - x^* \|^2_{Z^\dagger}}{k^2}$\\
\hline
smooth strong convex & $\left( 1 - \sigma_Z \right)^k$  & $\left( 1 - \sqrt{\frac{\sigma_Z}{\nu}} \right)^k$ \\
\hline
\end{tabular}
\end{table}
Table~\ref{tbl:comparison} summarizes the convergence rates in $\E[f(x^k)] - f^*$ of \RSD and A-\RSD algorithms for smooth (strongly) convex objective functions (we assumed for simplicity that ${\cal R}^2(x^0) \leq \|x^0 - x^* \|^2_{Z^\dagger}$). We observe from this table that we have obtained the typical convergence rates for these two methods, in particular, A-\RSD converges with one order of magnitude faster than \RSD\!, see  \cite{Nes:10} for more details.  It is important to note that in this work we provide the first analysis of an accelerated random sketch descent (A-\RSD) algorithm for optimization problems with multiple non-separable linear constraints.


\section{Illustrative numerical experiments}
In this section we provide several numerical examples  showing the  benefits of random sketching and the performances of our new algorithms.

\vskip10pt
\noindent {\bf Experiment \#1: A pre-fixed coordinate sampling can be a disaster.}
Recently, in \cite{TuVen:17} it has been shown for linear systems that Gauss-Seidel algorithm with randomly sampled coordinates substantially outperforms Gauss-Seidel with any fixed partitioning of the coordinates that are chosen ahead of time. Motivated by this finding, we also  analyze the
behavior of \RSD and A-\RSD algorithms for fixed coordinate sketch, random coordinate sketch and
Gaussian sketch. We build two challenging problems. One problem has a particular structure with a single linear constraint.  The second example is easier, and it involves a random matrix, but the linear constraints  make the problem  more challenging. The first problem is to minimize
the following convex optimization problem parameterized by $\delta \in [0,1]$:
\begin{equation}
\min_{x \in \R^n} x^T \left(I_n + (1-\delta)(e_1e_n^T+ e_ne_1^T)\right)x  \qquad \mbox{s.t.} \quad  e^Tx=0.
\end{equation}
We consider three different choices for $S$, fixed partition of the coordinates, random partition of the coordinates, and a Gaussian random sketch:
\begin{align*}
\text{fixed}:& \;\;  S_{(i,i+1)} = [e_i \; e_{i+1}] \quad \forall i=1:n-1 \\
\text{random}:& \;\; S_{(i,j)} = [e_i \; e_j] \quad \forall i < j\\
\text{Gaussian}: & \;\; S  = [\mathcal{N}(0,1)]^{n\times 2},
\end{align*}
where we recall that $e_i$  denotes the $i$th column of the identity matrix $I_n$ and
$\mathcal{N}(0,1)$ is normally distributed random variable with mean $0$ and variance $1$.
We use the same sketching also for the second problem, where  $M = M_0  + \delta I_n$ and $M_0 \succeq 0$ is a rank deficient random matrix,  and $f(x) = \frac12 x^T M x$. In this case we denote $\{v_k\}_{k=0}^n$ to be a set of orthogonal eigenvectors of $M$, such that $v_1$ corresponds to the largest eigenvalue and $v_n$ is the eigenvector which corresponds to the smallest eigenvalue. We have chosen $x^0 = v_1$ and $A =  v_2^T$. The optimal solution for both problems is $x^* = {\bf 0}$ with $f(x^*) = 0$.

\begin{figure}
\centering
\includegraphics[scale=0.45]{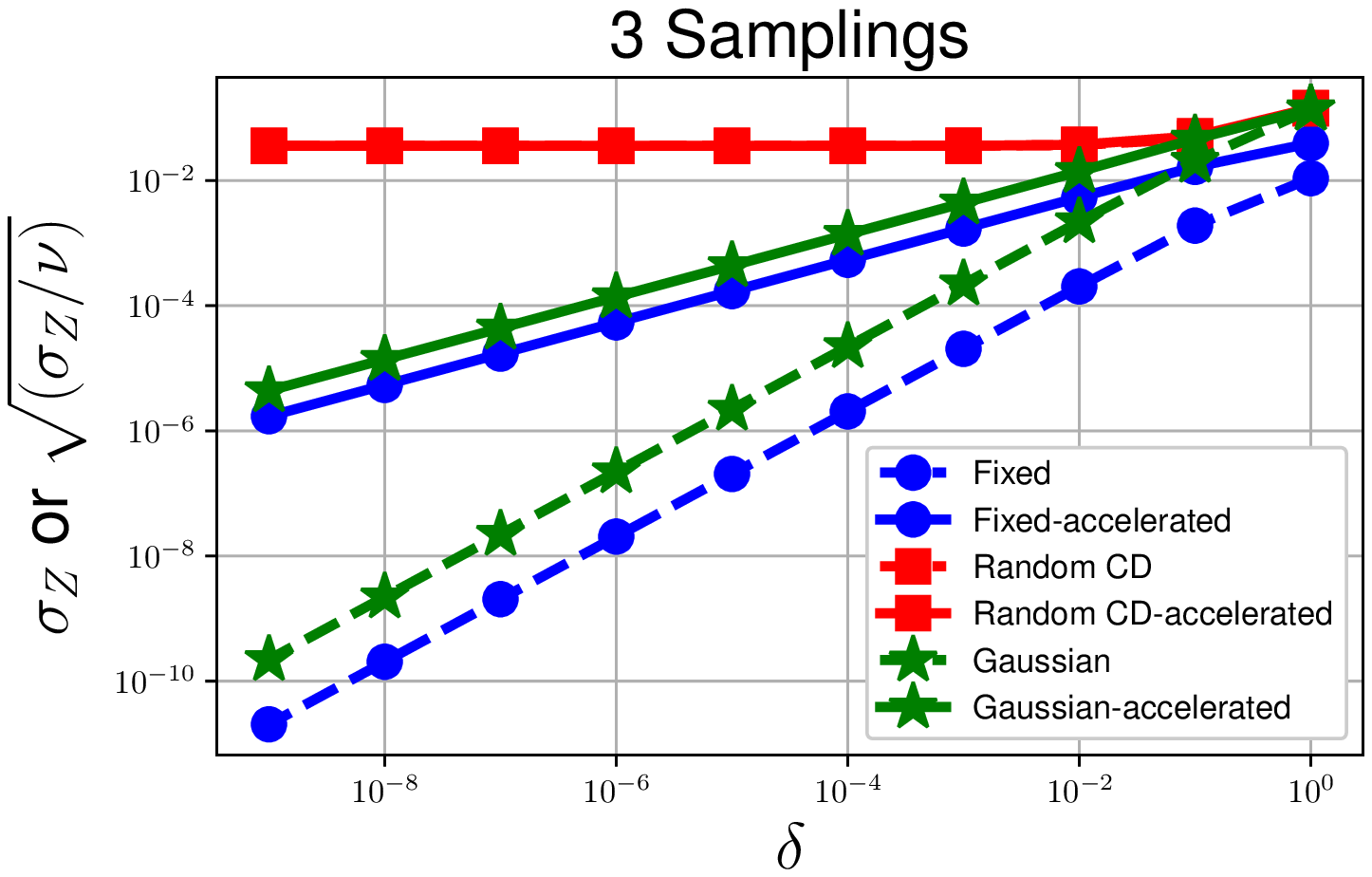}
\includegraphics[scale=0.45]{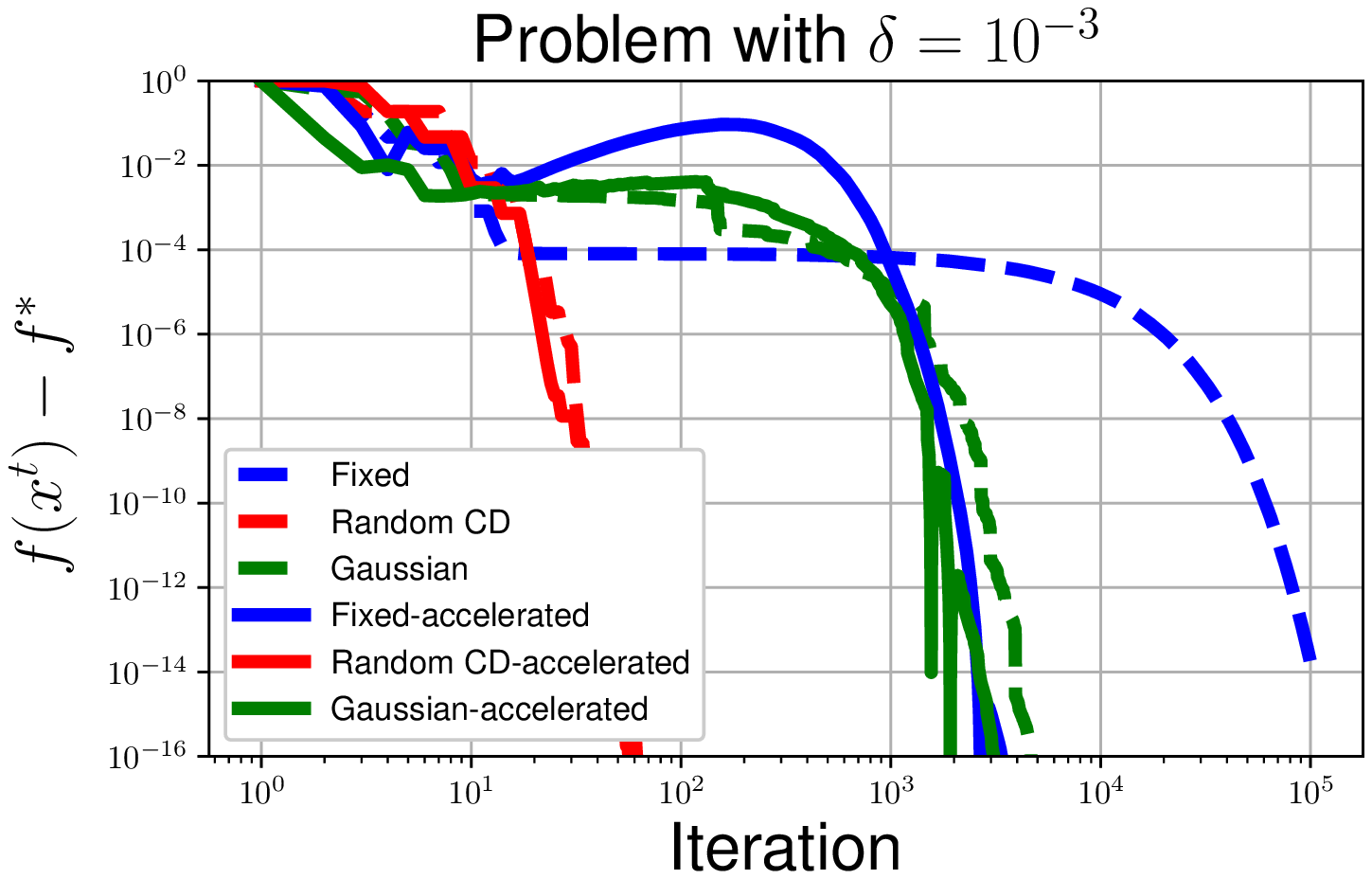}
\includegraphics[scale=0.45]{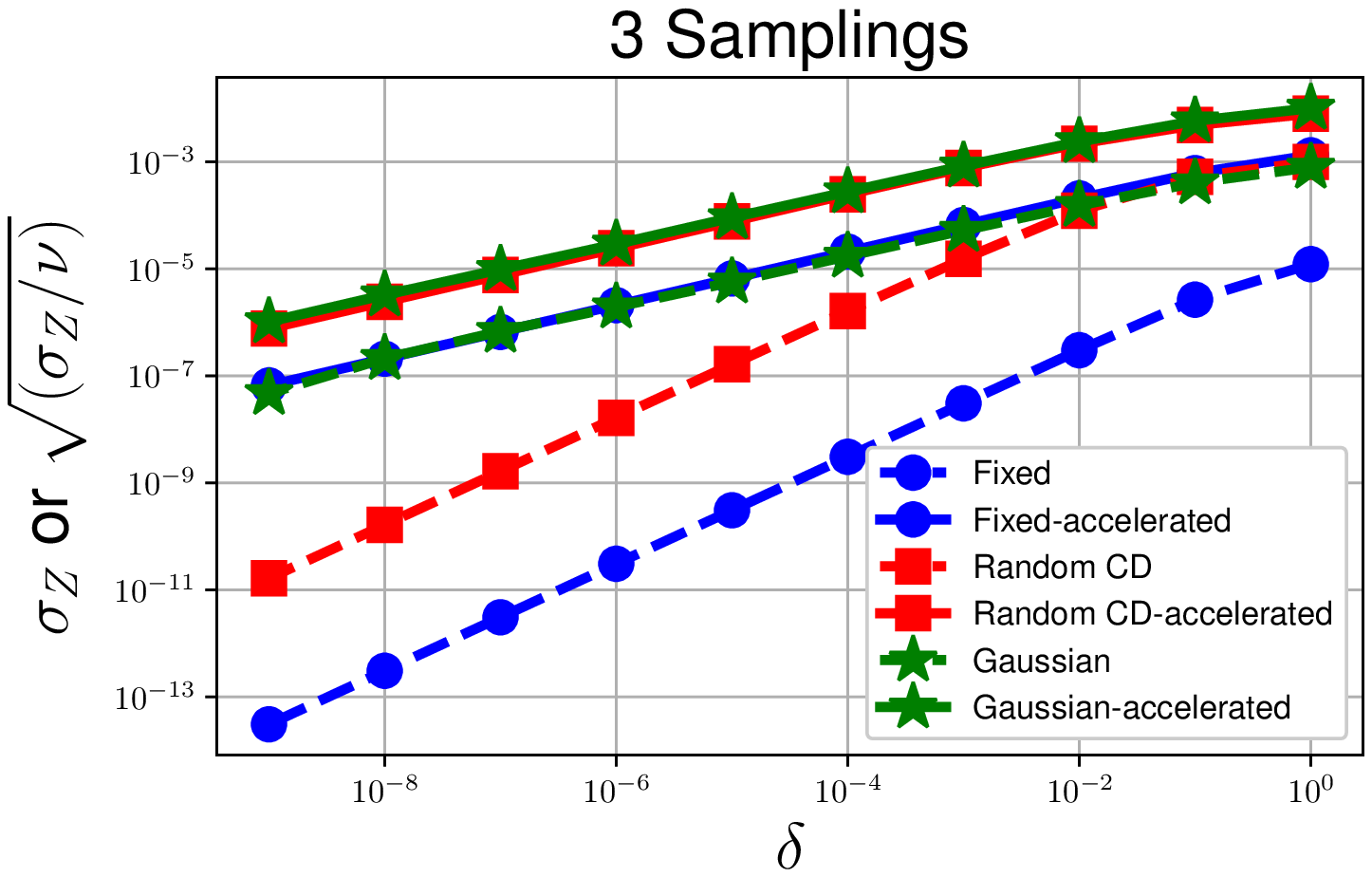}
\includegraphics[scale=0.45]{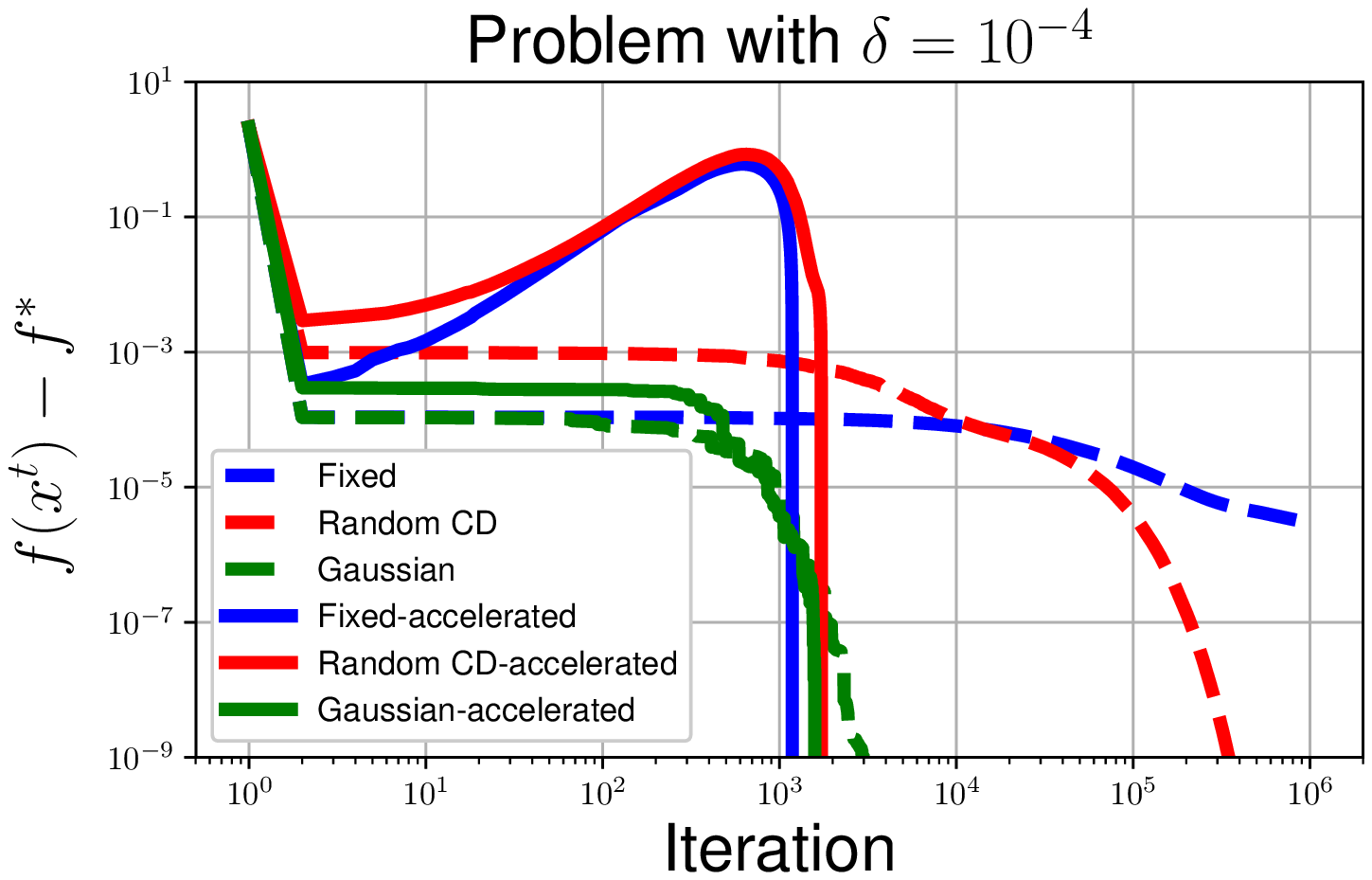}
\caption{Behavior of \RSD and A-\RSD for two problems and 3 different random sketch samplings.}
\label{fig:fixedVsRandom}
\end{figure}
In Figure~\ref{fig:fixedVsRandom}
top row, we show the results for the first problem and in the second row we show the results for the second problem. On the left, we show the important quantities $\sigma_Z$ or $\sqrt{\sigma_Z / \nu }$ which characterize the convergence rates of the two algorithms in the strongly convex case (see Table~\ref{tbl:comparison}). In the right column we show the typical evolution of $f(x^k)-f^*$. One can observe that for the first problem, the random  sampling is the best both in practice, whereas the other two samplings are suffering. The main reason is that for this problem, the most important sketch matrix $S$ is $S = [e_1 \; e_n]$ which is selected more often by the random sketching than the other two sketching  strategies. For the second test problem the best sketching is the Gaussian sketch. Therefore, empirically this experiment shows that both algorithms, \RSD  and A-\RSD, based on random sketching provides speedups  compared to  the fixed partitioning counterpart.

\vskip10pt
\noindent
{\bf Experiment \#2: The effect of a quadratic upper-bound in convergence speed.} In this experiment, we investigate the benefit of using the full matrix $M$  in \eqref{eq:M}
as compared to  just using a scaled diagonal upper-bound as considered e.g. in \cite{FroRei:15,Nec:13}. Consider the following convex optimization problem parameterized by 
$\delta \in [0,1]$:
\begin{equation}
\min_{x \in \R^n} x^T  \underbrace{(\delta I_n + (1-\delta) ee^T)}_{B\succeq 0} x
\quad \mbox{s.t.} \quad  e^Tx = 0.
\end{equation}

\begin{figure}
\centering
\includegraphics[scale=0.45]{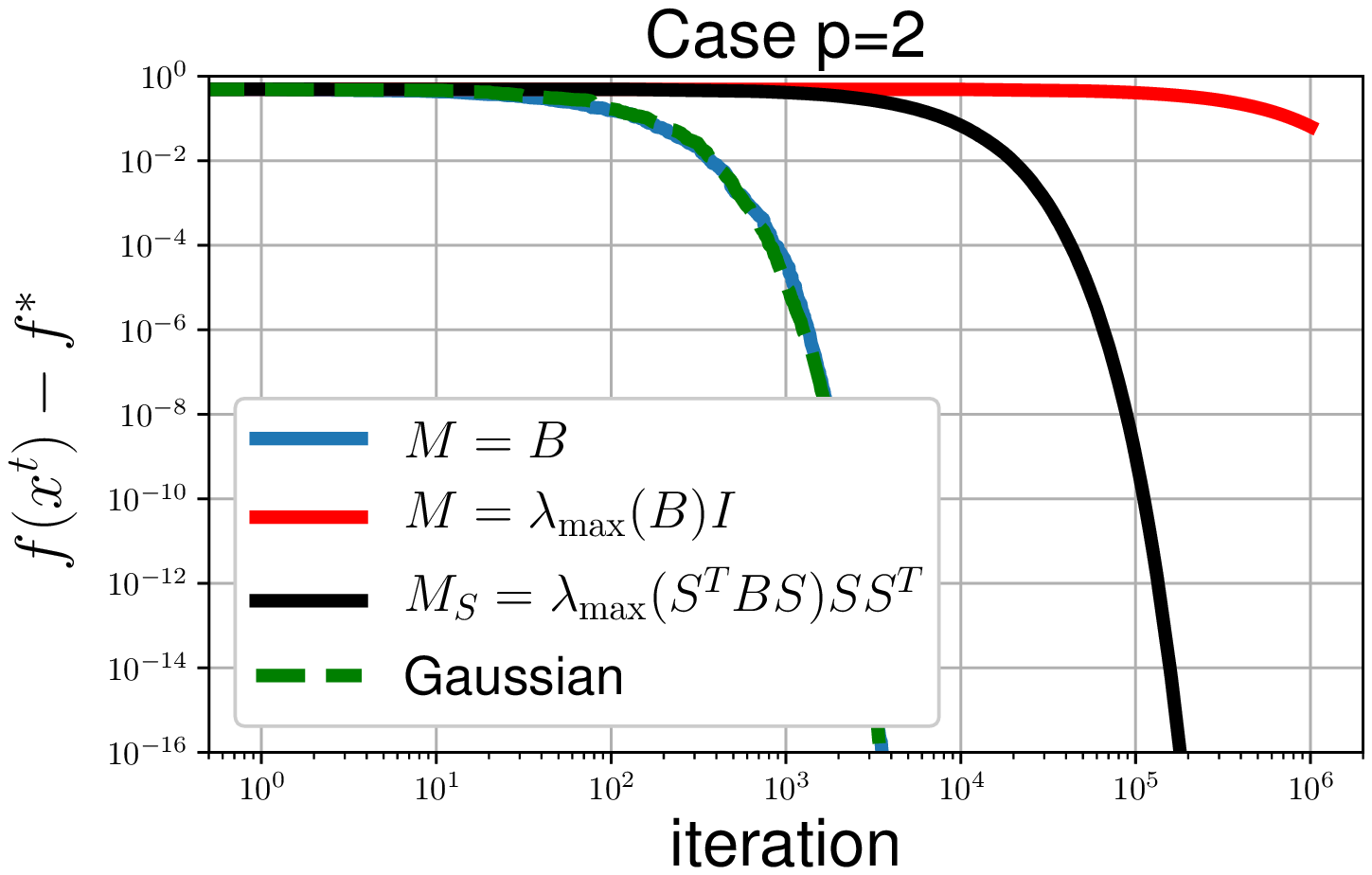}
\includegraphics[scale=0.45]{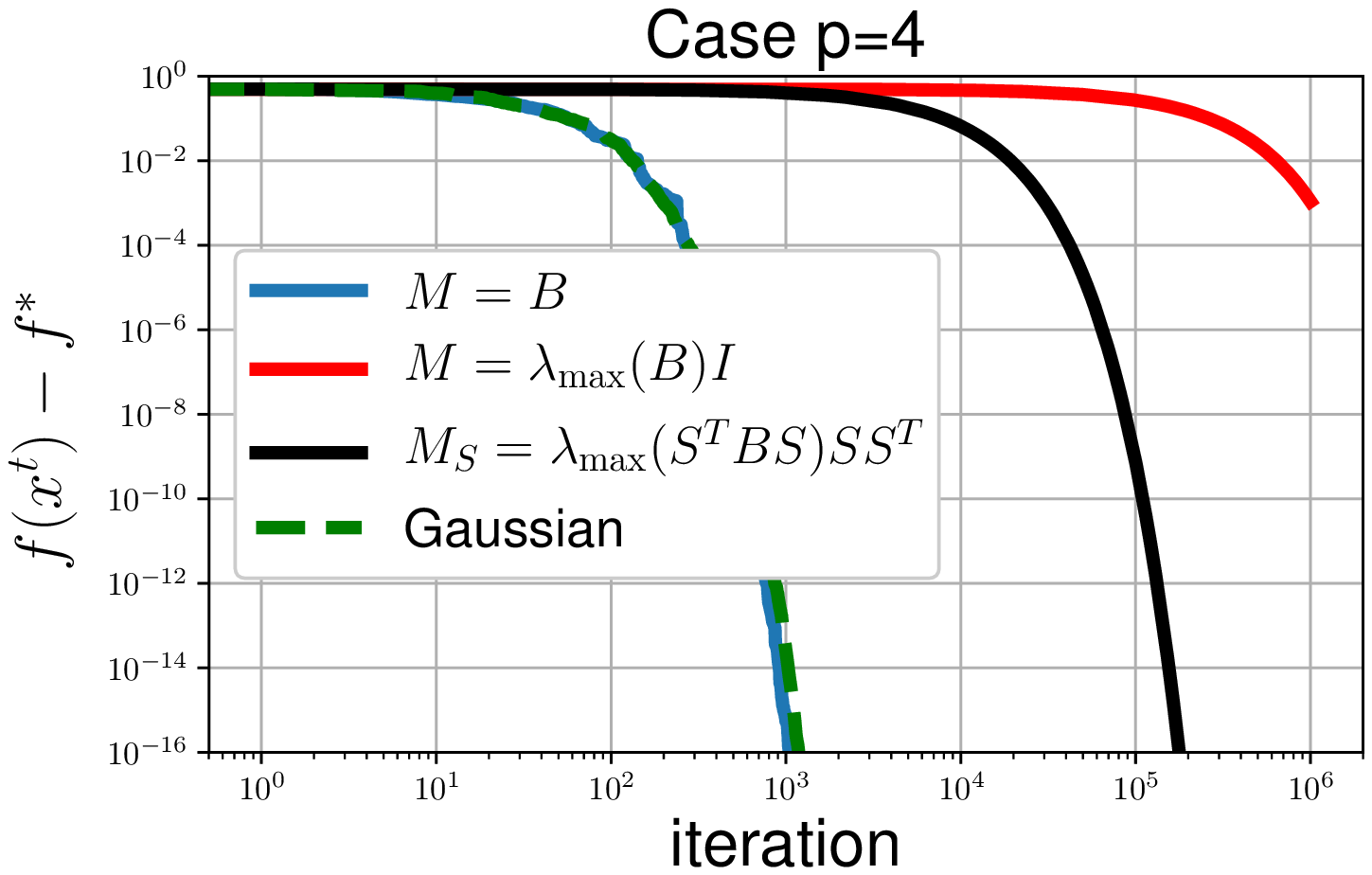}
\includegraphics[scale=0.45]{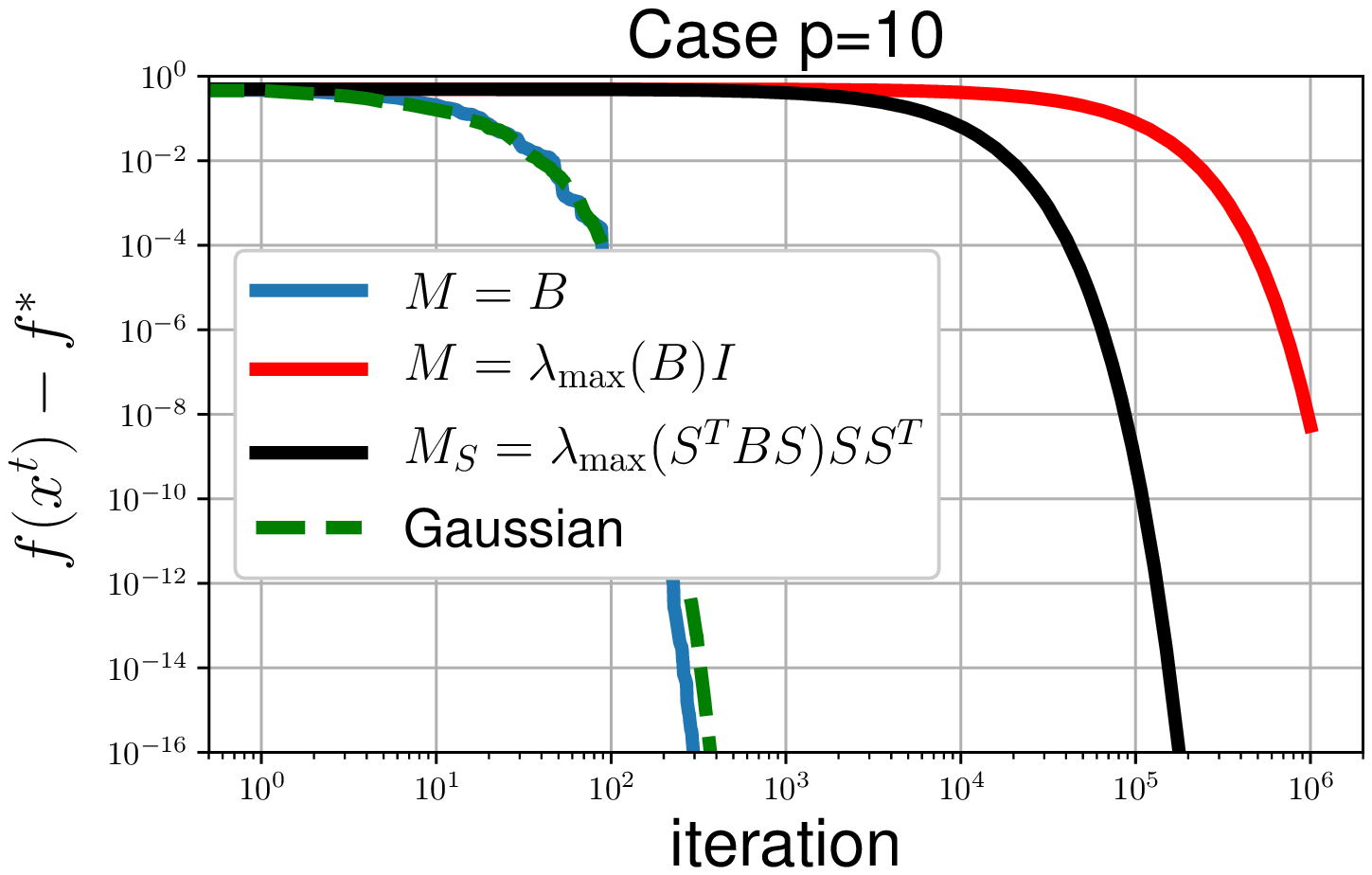}
\includegraphics[scale=0.45]{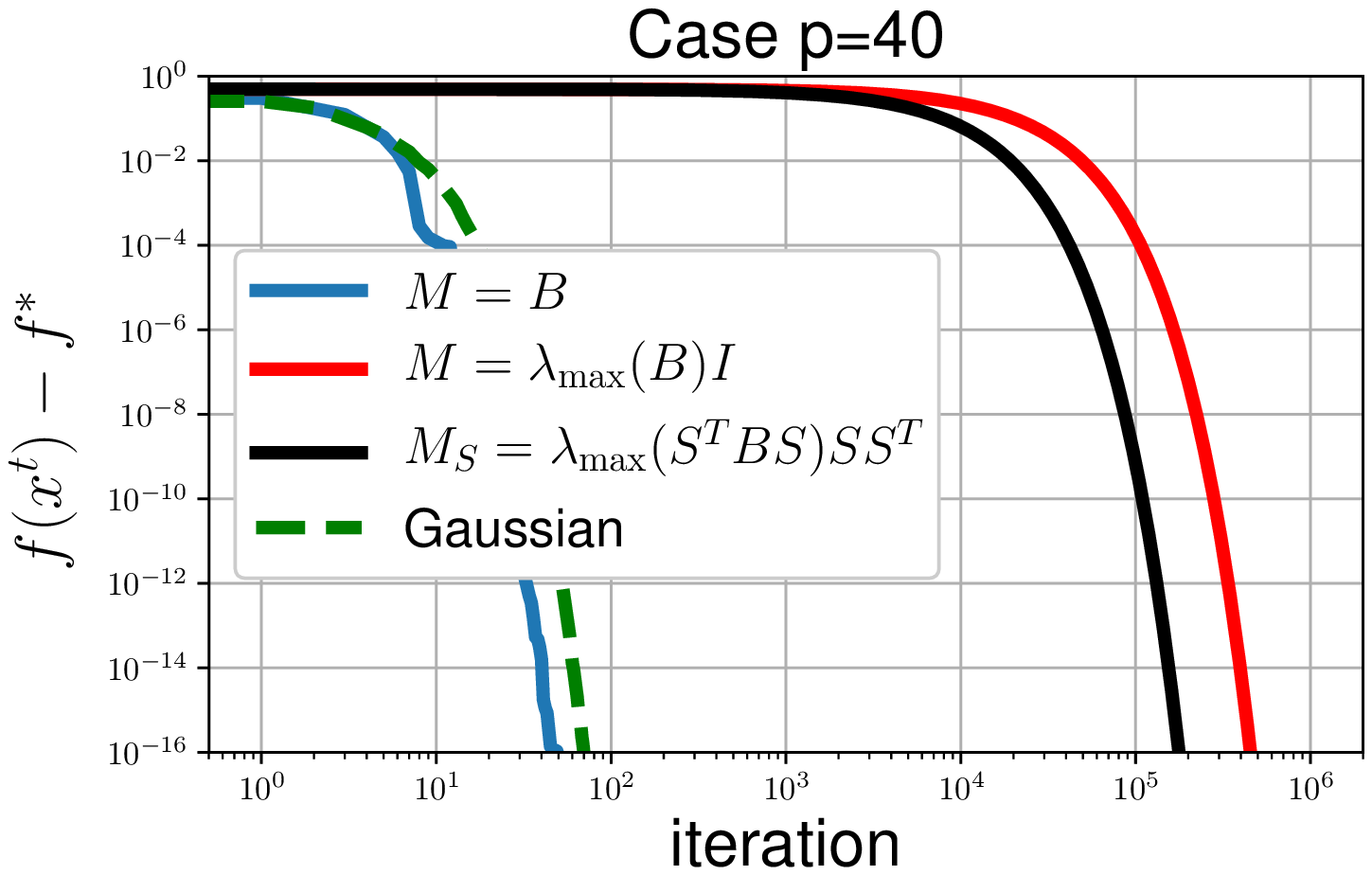}
\caption{Comparison of speed of Algorithm \RSD for various choices of $M$ and $p$.}
\label{fig:MvsLs}
\end{figure}
We  compare the speed of Algorithm \RSD when the random matrix $S$ is chosen uniformly at random as $p$  columns of the identity matrix and consider three choices for the matrix $M$:  $M = B$, $M = \lambda_{\max}(B) I_n$ and $M_S = \lambda_{\max}( S^T B S) S S^T$. We also implement  \RSD for  the Gaussian sketch and $M=B$. From Figure \ref{fig:MvsLs} one can observe that if we set $M=B$, then increasing $p$ will  decrease the number of iterations needed to achieve the desired accuracy with the best rate. We can also observe that Gaussian sketch or coordinate descent sketch have  a similar behavior for the case $M=B$.

\vskip10pt
\noindent
{\bf Experiment \#3: Portfolio optimization with specified industry allocation.}
In Section~\ref{sec:po} we have described the basic Markowitz portfolio selection model
\cite{Mar:52}. We have also described a variant of the basic model which assumes that
investor also decide how much net wealth would be allocated in different asset classes (e.g. Financials, Health Care, Industrials, etc). When we have $C$ asset classes, then the problem of minimizing the risk with all the desired constraints will lead to $C+2$ linear constraints. In Figure~\ref{fig:Portfolio} we show the performance of the \RSD algorithm with the sketch matrix $S$  chosen random Gaussian. We considered real data from the index S\&P500 which contains 500 assets split across $C = 11$ asset classes. The $\mu_i$s and $\Sigma$ were estimated from the historical data. In the left plot we show the evolution of error $f(x^k)-f^*$ for various sizes of $p$ as a function of iterations and in the middle plot  we put on x-axis the computational time.  We can observe that increasing $p$ leads to significantly  decrease of number of iterations and also faster convergence in terms of wall-clock time. Note also that as $p$ is becoming larger, the per-iteration computational  cost increases  moderately (see the right plot).
\begin{figure}
\centering
\includegraphics[scale=0.32]{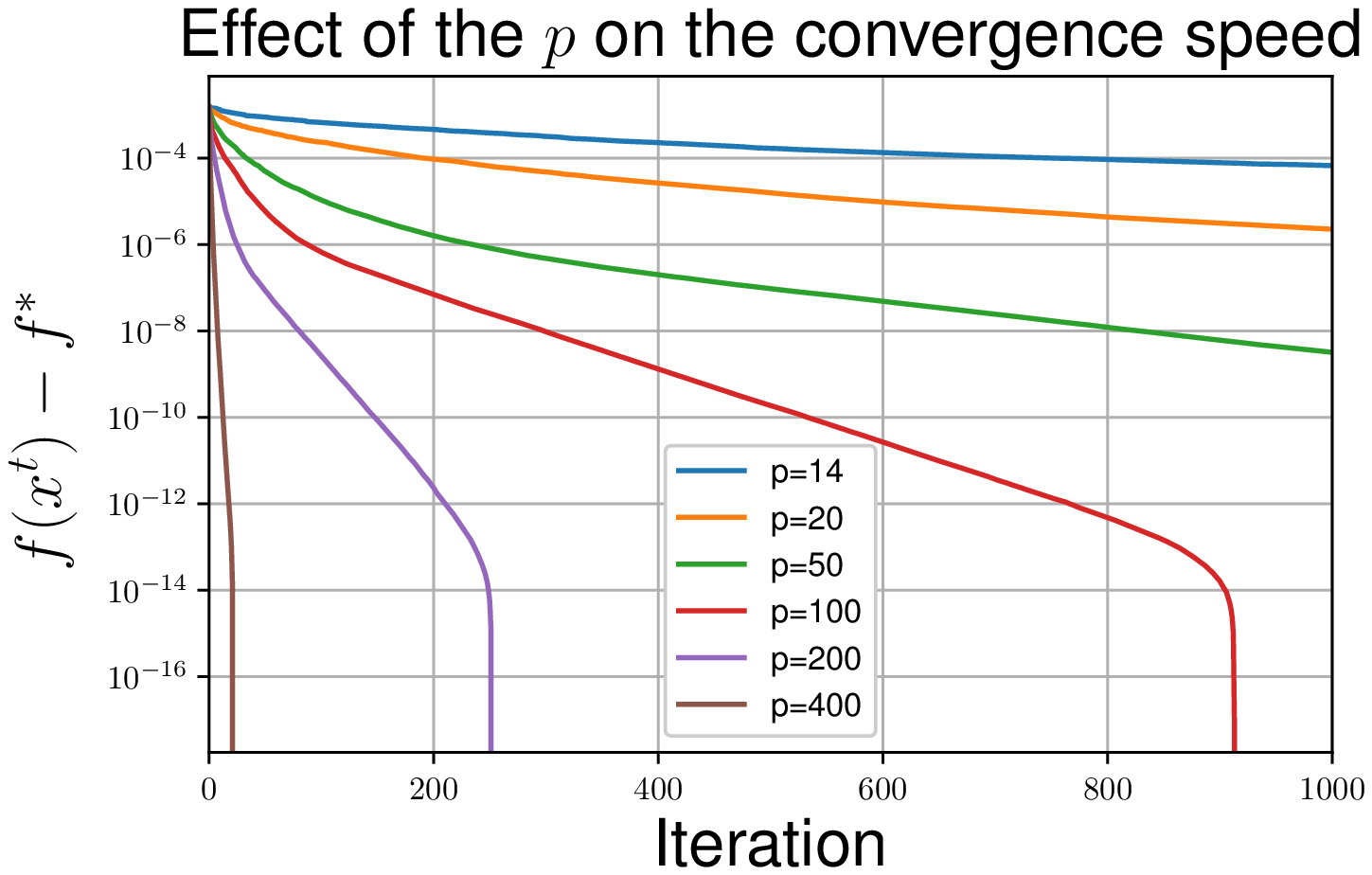}
\includegraphics[scale=0.32]{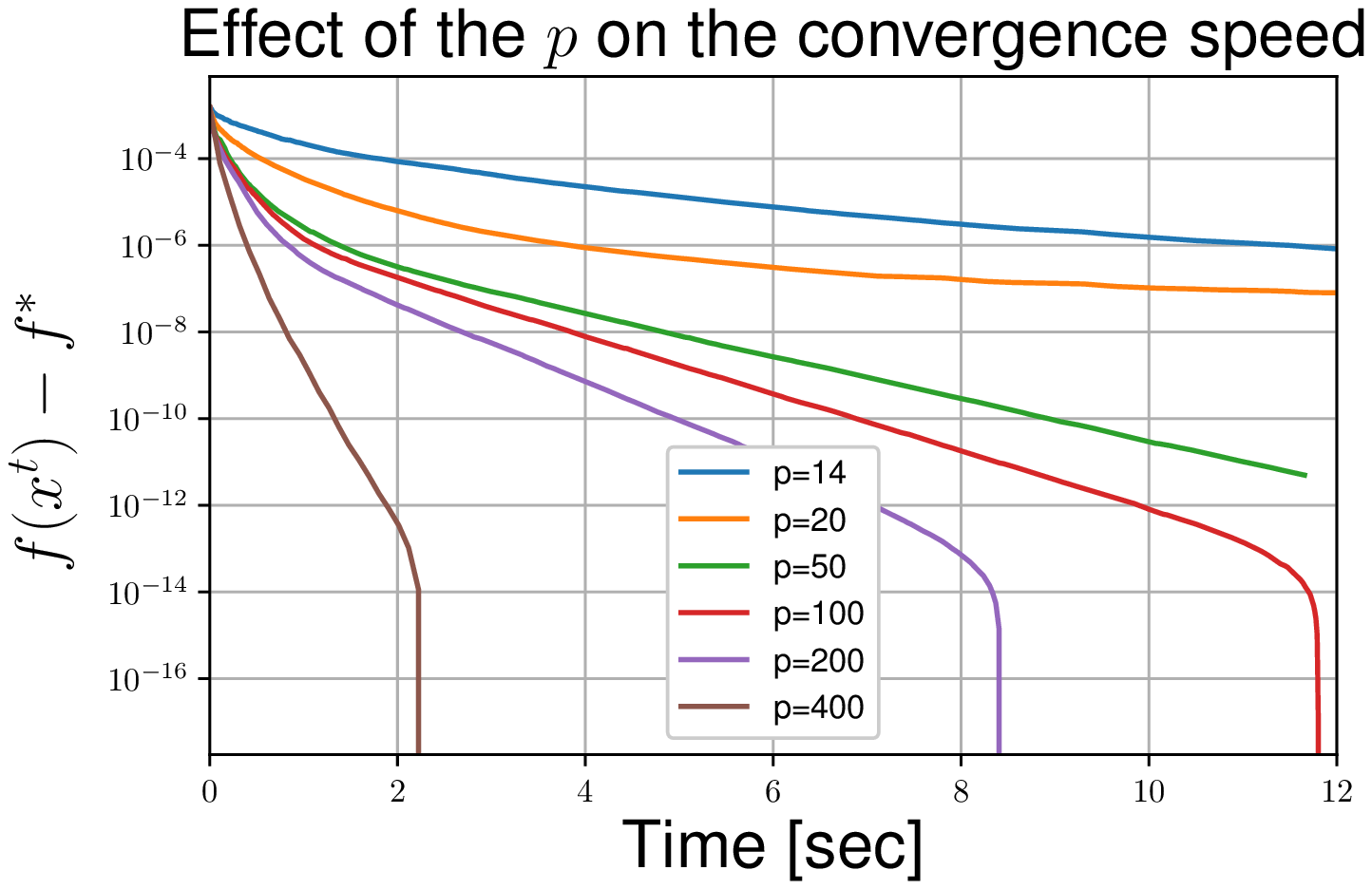}
\includegraphics[scale=0.32]{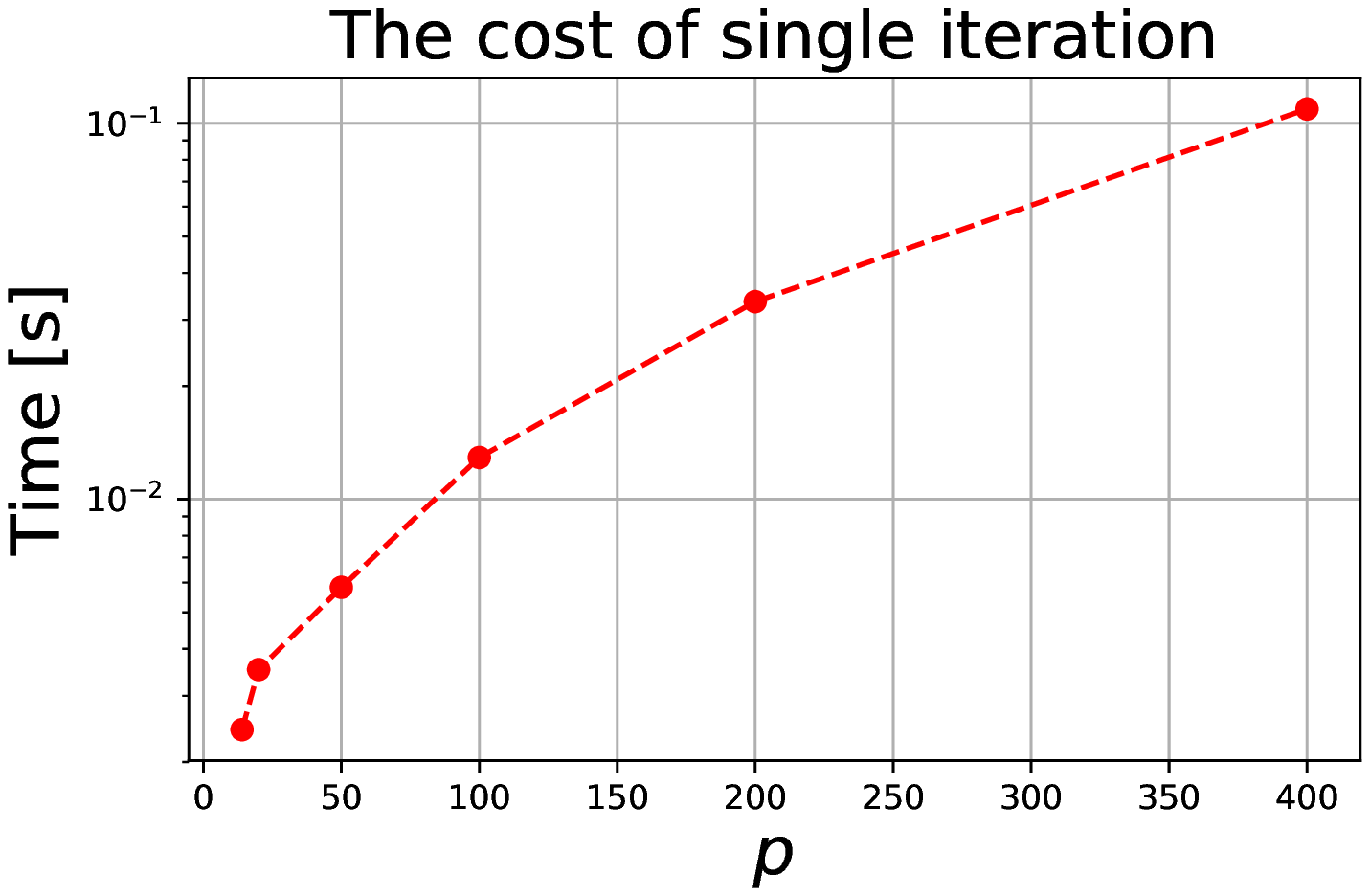}
\caption{Algorithm \RSD on  Markovitz portfolio optimization problem with $13$ linear  constraints.}
\label{fig:Portfolio}
\end{figure}


\section*{Acknowledgements}
The work of Ion Necoara was supported by the Executive Agency for
Higher Education, Research and Innovation Funding (UEFISCDI), Romania, under PNIII-P4-PCE-2016-0731, project ScaleFreeNet, no. 39/2017. The work of  Martin Tak\'a\v{c} was partially supported by the U.S. National Science Foundation, under award numbers NSF:CCF:1618717, NSF:CMMI:1663256 and NSF:CCF:1740796.


\bibliographystyle{plain}
\bibliography{citations}


\appendix
\section{\!\!\!\!}
In this appendix we discuss how to implement the A-\RSD updates without  full-dimensional vector operations. Recall that we assume the following settings: the sketch matrix  $S$ is sparse and  we can efficiently evaluate $\nabla f(\alpha v + \beta u)$. First we derive an efficient  implementation of A-\RSD iterations for strongly convex objective functions and then a simplified implementation for the convex case. Following a similar approach as in the coordinate descent work proposed in \cite{Lee:Sidford} for solving linear systems and further  extended in \cite{FerRic:15} for accelerated coordinate descent method with separable composite problems we note that:
\begin{align*}
y^{k+1} &= \alpha_{k+1} v^{k+1} + (1-\alpha_{k+1}) x^{k+1}\\
&= (1 -\alpha_{k+1} \beta_k) y^k + \alpha_{k+1} \beta_k v^k - (1- \alpha_{k+1} (1- \gamma_k)) Z_S \nabla f(y^k).
\end{align*}
Hence, we obtain the following recursion:
\begin{align}
\begin{pmatrix}
y^{k+1} \\
v^{k+1}
\end{pmatrix}
 = A_{k}
\begin{pmatrix}
y^k \\
v^{k}
\end{pmatrix}
- s_k,
\label{eq:afeawefawwa}
\end{align}
with
\begin{align*}
A_k = \begin{pmatrix}
 1  -\alpha_{k+1} \beta_k
 & \alpha_{k+1} \beta_k \\
  1-\beta_k & \beta_k
\end{pmatrix},
\qquad
s_k=\begin{pmatrix}
(1-\alpha_{k+1} (1- \gamma_k))Z_S \nabla f(y^k)\\
\gamma_k Z_S \nabla f(y^k)
\end{pmatrix}.
\end{align*}
Now, our goal is to maintain two sequences $\{u^k\}_k,\{w^k\}_k$
such that:
$\begin{pmatrix}
y^{k} \\ 
v^k
\end{pmatrix}
= B_k
\begin{pmatrix}
u^k \\ 
w^k
\end{pmatrix}$. Therefore, it has to hold that
\begin{align*}
B_{k+1}
\begin{pmatrix}
u^{k+1} \\ 
w^{k+1}
\end{pmatrix}
=
\begin{pmatrix}
y^{k+1} \\ 
v^{k+1}
\end{pmatrix}
\overset{\eqref{eq:afeawefawwa}}{=}
A_{k} B_k
\begin{pmatrix}
u^k \\ 
w^k
\end{pmatrix}
- s_k,
\end{align*}
and therefore we require
\begin{align*}
\begin{pmatrix}
u^{k+1} \\ 
w^{k+1}
\end{pmatrix}
=B_{k+1}^{-1} A_{k} B_k
\begin{pmatrix}
u^k \\ 
w^k
\end{pmatrix}
- B_{k+1}^{-1} s_k.
\end{align*}
In order to make this computationally efficient, it is sufficient to define $B_k$ recursively as:
$$ B_0 = I_2, \quad B_{k+1} = A_k B_k,$$
$u^0 = y^0$ and  $w^0 = v^0$ to obtain the following update rule
\begin{align*}
\begin{pmatrix}
u^{k+1} \\ 
w^{k+1}
\end{pmatrix}
=
\begin{pmatrix}
u^k \\ 
w^k
\end{pmatrix}
- B_{k+1}^{-1} s_k,
\end{align*}
which is a sparse update provided that $s_k$ is a sparse vector. However, when the sketch matrix 
$S$ is sparse the vector $Z_S \nabla f(y^k)$ is  sparse as well and consequently  $s_k$ is also a sparse vector (see Example \ref{example} where for $S_{(i,j)}=[e_i \ e_j]$ the corresponding vector $Z_{(i,j)} \nabla f(y)$ has only two non-zero entries). The final algorithm is  depicted in Algorithm~\ref{alg:arcd-efficient} below.

\begin{algorithm}
\caption{Efficient implementation of A-\RSD for sparse sketching: strongly convex case}
\label{alg:arcd-efficient}
\begin{algorithmic}[1]
\STATE {\bf Input:} Positive sequences $\{\alpha_k\}_{k=0}^\infty,  \{\beta_k\}_{k=0}^\infty,  \{\gamma_k\}_{k=0}^\infty$
\STATE choose  $x^0 \in \R^n$ such that $ A x^0 = b$ and set $u^0 = w^0 = x^0$
\STATE set $B_0 = I_2$
\FOR {  $k \geq 0$}
\STATE
sample $S \sim {\sS}$
\STATE compute $g = Z_S \nabla f\left( \underbrace{B_k^{11} u^k + B_k^{12} w^k }_{y^k} \right) $
\STATE
 $B_{k+1} = A_k B_k$
\STATE
$
\begin{pmatrix}
u^{k+1} \\ w^{k+1}
\end{pmatrix}
=
\begin{pmatrix}
u^k \\ w^k
\end{pmatrix}
-
B_{k+1}^{-1}
\begin{pmatrix}
(1-\alpha_{k+1} (1- \gamma_k)) g\\
\gamma_k g
\end{pmatrix}
$
\ENDFOR
\end{algorithmic}
\end{algorithm}


\vskip10pt
\noindent
{\bf Simplified Convex Case.}
In the case of non-strongly  convex objective function, the implementation can be significantly simplified using the fact that  $\beta_k = 1$ for all $k$. Then, we have:
\begin{equation}
v^{k+1} =  v^k  - \gamma_k Z_S \nabla f(y^k)
\label{eq:afwefawfawfcaw}
\end{equation}
and
\begin{align*}
y^{k+1} - v^{k+1} & = \alpha_{k+1} v^{k+1} + (1-\alpha_{k+1}) x^{k+1} - v^{k+1}\\
& = (1-\alpha_{k+1}) (y^k -  Z_S \nabla f(y^k) -v^k  + \gamma_k Z_S \nabla f(y^k))\\
&= (1-\alpha_{k+1}) (y^k-v^k) - (1-\alpha_{k+1})(1 - \gamma_k) Z_S \nabla f(y^k).
\end{align*}
Therefore, we obtain the following recursion:
\begin{align}
\begin{pmatrix}
y^{k+1}-v^{k+1} \\
v^{k+1}
\end{pmatrix}
 =
\tilde A_k
\begin{pmatrix}
y^k - v^{k} \\
v^{k}
\end{pmatrix}
- \tilde s_k,
\label{eqfafwafwea}
\end{align}
with
$$ \tilde A_k =
\begin{pmatrix}
1-\alpha_{k+1} & 0 \\
0 & 1
\end{pmatrix},
\qquad
\tilde s_k =
\begin{pmatrix}
(1-\alpha_{k+1})(1-\gamma_k) Z_S \nabla f(y^k)\\
\gamma_k Z_S \nabla f(y^k)
\end{pmatrix}.$$
Now, we see that the update of $v^k$ given by \eqref{eq:afwefawfawfcaw} is sparse if $Z_S \nabla f(y^k)$ is sparse. Further, we want to express $y^{k+1} - v^{k+1} = b_{k+1} u^{k+1}$. Then, from \eqref{eqfafwafwea} we have:
\begin{align*}
b_{k+1} u^{k+1} &= y^{k+1} - v^{k+1} = (1-\alpha_{k+1}) (y^{k} - v^{k}) - (1-\alpha_{k+1})(1-\gamma_k) Z_S \nabla f(y^k) \\
&= (1-\alpha_{k+1}) b_k  u^{k} - (1-\alpha_{k+1})(1-\gamma_k) Z_S \nabla f(y^k).
\end{align*}
Therefore, if we define $b_{k+1} = (1-\alpha_{k+1}) b_k$, this will simplify to:
\begin{align*}
u^{k+1} &= u^{k} - \frac{(1-\alpha_{k+1})(1-\gamma_k)}{b_{k+1}} Z_S \nabla f(y^k) = u^{k} - \frac{1-\gamma_k}{b_{k}}  Z_S \nabla f(y^k).
\end{align*}
It follows that the update of $u^k$  is also sparse if $Z_S \nabla f(y^k)$ is sparse. Next, we can easily compute $y^k = v^k + b_k u^k$ (however, this shouldn't be formed during the run of the algorithm). Finally, it is sufficient to note that $v^0 = x^0$, $u^0 = {\bf 0}$ and we can choose $b_0 = 1$.  The final algorithm is given in  Algorithm~\ref{alg:efficient:convexCase} below.

\begin{algorithm}
\caption{Efficient implementation of A-\RSD for sparse sketching: convex case}
\label{alg:efficient:convexCase}
\begin{algorithmic}[1]
\STATE {\bf Input:} Positive sequences $\{\alpha_k\}_{k=0}^\infty,   \{\gamma_k\}_{k=0}^\infty$
\STATE Choose  $x^0 \in \R^n$ such that $ A x^0 = b$
\STATE set $v^0 = x^0$, $u^0 = {\bf 0}$ and  $b_0 = 1$
\FOR {  $k \geq 0$}
\STATE sample $S \sim {\sS}$
\STATE compute $g= Z_S \nabla f \left(\underbrace{v^k + b_k u^k}_{y^k}\right)$
\STATE $v^{k+1} =  v^k  - \gamma_k g$
\STATE $u^{k+1} = u^{k} - \frac{1-\gamma_k}{b_{k}} g$
\STATE $b_{k+1} = (1-\alpha_{k+1}) b_k$
\ENDFOR
\end{algorithmic}
\end{algorithm}

\end{document}